\documentclass[12pt, a4paper]{amsart}
\usepackage[left=3cm, right=3cm]{geometry}
\usepackage{amsmath,amssymb,amsthm,enumerate}
\usepackage{graphicx}
\usepackage{xypic} 
\usepackage{mathrsfs} 
\usepackage{xspace}
\usepackage[pagebackref=true]{hyperref}

\title[Indicability, residual finiteness, and simple subquotients]{Indicability, residual finiteness, and simple subquotients of groups acting on trees}

\author[P.-E. Caprace]{Pierre-Emmanuel Caprace}
\thanks{P.-E.C. is a F.R.S.-FNRS senior research associate, supported in part by EPSRC grant no EP/K032208/1.}
\address{Universit\'e catholique de Louvain, IRMP, Chemin du Cyclotron 2, bte L7.01.02, 1348 Louvain-la-Neuve, Belgique}
\email{pe.caprace@uclouvain.be}
\author[P. Wesolek]{Phillip Wesolek}
\address{Binghamton University, Department of Mathematical Sciences, PO Box 6000, 	Binghamton, New York 13902-6000, USA}
\email{pwesolek@binghamton.edu}\date{July 26, 2017}

\newtheorem{thm}{Theorem}[section]

\newtheorem{prop}[thm]{Proposition}
\newtheorem{lem}[thm]{Lemma}
\newtheorem{cor}[thm]{Corollary}

\theoremstyle{definition}
\newtheorem{defn}[thm]{Definition}
\newtheorem{qu}[thm]{Question}
\newtheorem{rmk}[thm]{Remark}

\newcommand{\Zbb}{\mathbf{Z}}
\newcommand{\Nb}{\mathbf{N}}

\newcommand{\mc}[1]{\mathcal{#1}}

\newcommand{\ms}[1]{\mathscr{#1}} 

\newcommand{\acts}{\curvearrowright}

\newcommand{\tdlcsc}{t.d.l.c.s.c.\@\xspace}
\newcommand{\Sym}{\mathrm{Sym}}
\newcommand{\Aut}{\mathrm{Aut}}

\DeclareMathOperator{\Ker}{Ker}
\newcommand{\Res}{\mathrm{Res}}
\newcommand{\rest}{\upharpoonright}
\newcommand{\cent}{\mathrm{C}}
\newcommand{\norm}{\mathrm{N}}
\newcommand{\con}{\mathrm{con}}
\newcommand{\triv}{\{1\}}
\newcommand{\QZ}{\mathrm{QZ}}
\newcommand{\Sub}{\mathbf{Sub}}
\newcommand{\grp}[1]{\langle #1 \rangle}

\newcommand{\cgrp}[1]{\ol{\langle #1 \rangle}}
\newcommand{\ol}[1]{\overline{#1}}

\newcommand{\wh}[1]{\widehat{#1}}

\usepackage[usenames, dvipsnames]{color}

\begin{document}

\begin{abstract}
We establish three independent results on groups acting on trees. The first implies that a compactly generated locally compact group which acts continuously  on a locally finite tree with nilpotent local action and no global fixed point is virtually indicable; that is to say, it has a finite index subgroup which surjects onto $\mathbf{Z}$. The second ensures that  irreducible cocompact lattices in a product of non-discrete locally compact groups such that one of the factors acts vertex-transitively on a tree with a nilpotent  local action cannot be residually finite. This is derived from a general result, of independent interest,  on irreducible lattices in product groups.  The third implies that every non-discrete Burger--Mozes universal group of automorphisms of a tree with an arbitrary prescribed local action admits a compactly generated closed subgroup with a non-discrete simple quotient. As applications, we answer a question of D.~Wise by proving the non-residual finiteness of a certain lattice in a product of two regular trees, and we obtain a negative answer to   a question of C. Reid, concerning the structure theory of locally compact groups. 
\end{abstract}
\maketitle

\section{Introduction}

Given a group $G$ acting by automorphisms on a graph $X$, the \textbf{local action} of $G$ at a vertex $v \in VX$ is the permutation group induced by the   action of the vertex stabilizer $G_{(v)}$ on the set of edges $E(v)$ emanating from $v$. Various results from the literature show how restrictions on the local action impact the global properties of the group $G$. This phenomenon is strikingly illustrated by the work of M. Burger and S. Mozes on lattices in products of trees; see \cite{BuMo,BuMo2}. 

Our first main result provides an illustration of this paradigm in the case that $X$ is a tree. A (topological) group is called \textbf{virtually indicable} if it has a finite index (finite index open) subgroup admitting a (continuous) surjective homomorphism onto the infinite cyclic group. A tree is called \textbf{leafless} if it has no vertex of valency~$1$. 

\begin{thm}[See Theorem~\ref{thm:Indic}]\label{thmintro:Fnilpotent}
Let $G$ be a topological  group with a  continuous, cocompact action  by automorphisms on an infinite locally finite leafless tree  $T$.  Suppose that   the local action $F(v)$ of $G$ at every vertex $v$ is such that the subgroup of $F(v)$ generated by its point stabilizers is intransitive on $E(v)$. Then $G$ is virtually indicable.
\end{thm}

For every finite transitive permutation group $F$ of degree $d$ that is generated by its point stabilizers, there exists an infinite simple group acting transitively on the (undirected) edges of the $d$-regular tree whose local action at every vertex is isomorphic to $F$; see \cite[Proposition~3.2.1]{BuMo} or the discussion preceding Theorem~\ref{thmintro:Subquotient} below. The condition on the local action in Theorem~\ref{thmintro:Fnilpotent}  can therefore not be weakened.  By considering an $\Aut(T_3)$-equivariant embedding of the trivalent tree  $T_3$ in the $4$-regular tree, one sees that the hypothesis of minimality of the $G$-action on $T$ is also necessary in the Theorem, even if $G$ is compactly generated. 

A special class of permutation groups satisfying the local condition of the theorem is that of nilpotent groups; see Lemma~\ref{lem:NilpPerm}. In that particular case, the minimality assumption on the $G$-action can be replaced by the assumption that $G$ be compactly generated, yielding the following result.

\begin{cor}\label{cor:nilpotentLocal}
Let $G$ be a compactly generated locally compact group with a  continuous action  by automorphisms on a locally finite tree  $T$.  If the local action of $G$ at every vertex is nilpotent, then either $G$ fixes a vertex or edge or $G$ is virtually indicable.
\end{cor}

A natural framework in which these results are relevant is that of lattices in products of trees. More generally, if $\Gamma$ is a finitely generated lattice in a product of the form $\Aut(T) \times H$ where $T$ is a locally finite tree and $H$ is a locally compact group, then $\Gamma$ is virtually indicable as soon as the local action of $\Gamma$ at every vertex of $T$ is nilpotent. Our second main result shows that in the latter situation the lattice $\Gamma$ cannot be residually finite, unless it is reducible. 

\begin{thm}[See Corollary~\ref{cor:NonRF}]\label{thmintro:LatticeNilAction}
Let $T$ be a locally finite leafless tree   and $H$ be a compactly generated totally disconnected locally compact group with a trivial amenable radical. Let   $\Gamma \leq \Aut(T) \times H$ be a cocompact lattice whose projection  to $H$ has non-discrete image. If the $\Gamma$-action on $T$ is vertex-transitive and the local action of $\Gamma$ at a vertex of $T$ is nilpotent,   then the projection of $\Gamma$ to $H$ is non-injective, and $\Gamma$ is not residually  finite. 
\end{thm}  

The condition of vertex-transitivity of $\Gamma$ on $T$ can be removed if one strengthens slightly the hypothesis on the local action, see Corollary~\ref{cor:NonRF}. 

Theorem~\ref{thmintro:LatticeNilAction} applies in particular to  D. Wise's iconic example of an irreducible lattice in $\Aut(T_4) \times \Aut(T_6)$ where $T_d$ denotes the $d$-regular tree, which we call the \textbf{Wise lattice}; see \cite[Example~4.1]{Wise} and \cite{Wise_PhD}. In \cite[Main Theorem~7.5]{Wise}, Wise proves that his lattice has an inseparable finitely generated subgroup, which he uses to prove that the double of $\Gamma$ over that subgroup, which is an irreducible lattice in $\Aut(T_8) \times \Aut(T_6)$, is not residually finite. We show that the Wise lattice itself already fails to be residually finite, thereby resolving a problem posed by D.~Wise \cite[Problem~10.19]{Wise}.

\begin{cor}\label{cor:WiseExample}
	The Wise lattice  $\Gamma \leq \Aut(T_4) \times \Aut(T_6)$ is not residually finite. 
\end{cor}
\begin{rmk}  In \cite{BK17}, Corollary~\ref{cor:WiseExample} is independently proved, via considering square complexes associated to automata.
\end{rmk}

Theorem~\ref{thmintro:LatticeNilAction} is deduced from a general statement on irreducible lattices in products of locally compact groups. A special case of which is the following (see also Theorem~\ref{thm:ResNoethMultiple} for another related result of independent interest). 

\begin{thm}[See Theorem~\ref{thm:MultipleProduct}]
Let $G = G_1 \times \dots \times G_n$ be a product of non-discrete compactly generated  totally disconnected locally compact groups that have a trivial amenable radical and no infinite discrete quotient. Let $\Gamma \leq G$ be a cocompact lattice whose projection to $G_i$ is dense for all $i$. If $\Gamma$ is residually finite, then every compact open subgroup of $G$ has a compact normalizer, and $G$ has a trivial quasi-center.
\end{thm}

For our last main result, we study the universal groups of Burger--Mozes; see  \cite[\S3.2]{BuMo}. These groups depend on a choice of a finite permutation group $F$ and are denoted by $U(F)$. For $F$ with degree $d$, the group $U(F)$ is a closed vertex-transitive subgroup of $\Aut(T_d)$. Whenever $F$ does not act freely, the subgroup generated by pointwise edge-stabilizers, denoted by $U(F)^+$, is an abstractly simple non-discrete closed subgroup of $\Aut(T_d)$. When $F$ is transitive and generated by its point stabilizers, the group $U(F)^+$ is a compactly generated non-discrete simple group acting edge-transitively on $T$; see Proposition~\ref{prop:BuMoSimple} below.  For such $F$, the group $U(F)^+$ therefore belongs to the following interesting class of groups.

\begin{defn} \label{def:Ss}
Let $\ms{S}$ denote the class of non-discrete totally disconnected locally compact groups that are topologically simple and compactly generated.
\end{defn}

When $F$ is not transitive or not generated by its point stabilizers, the group $U(F)^+$ is not compactly generated; e.g. see Corollary~\ref{cor:U(F)^+} below. We show that $U(F)$ nonetheless admits a group in $\ms{S}$ as a subquotient.

\begin{thm}[See Theorem~\ref{thm:involve_S}]\label{thmintro:Subquotient}
Let $F$ be a finite permutation group  which does not act freely. Then $U(F)$ has a compactly generated closed subgroup $H$ admitting a discrete normal subgroup $D$ such that $H/D$ is a non-discrete compactly generated simple group. 
\end{thm}

Combining Theorems~\ref{thmintro:Fnilpotent} and~\ref{thmintro:Subquotient}, we obtain a negative answer to a question asked by Colin Reid \cite[Question 2]{R16}; see Section~\ref{sec:TitsCore}.

\section*{Acknowledgment}

We  thank the Isaac Newton Institute for Mathematical Sciences, Cambridge for support and hospitality during the programme \textit{Non-positive curvature group actions and cohomology} where part of the work on this paper was accomplished. We are grateful to Jingyin Huang for drawing our attention to \cite[Lemma~9.2]{JH16}, which has been a source of inspiration for the proof of Theorem~\ref{thm:involve_S}. We also thank Marc Burger for his interest in this work and for a pleasant conversation during which Lemma~\ref{lem:QZ} came to life. We finally thank Colin Reid for his helpful suggestions of examples of Frobenius groups.

\section{Preliminaries}

\subsection{Graphs and Bass--Serre Theory}

Following J.-P. Serre, cf. \cite{Serre80}, a \textbf{graph} is a tuple $\Gamma = (V\Gamma,E\Gamma,o,r)$ consisting of a vertex set $V\Gamma$, a directed edge set $E\Gamma$, a map $o:E \rightarrow V$ assigning to each edge an \textbf{initial vertex}, and a bijection $r: E \rightarrow E$, denoted $e \mapsto \ol{e}$ and called \textbf{edge reversal}, such that $r^2 = \mathrm{id}$ and $e\neq \ol{e}$. Edge reversal and the initial vertex map together give a terminal vertex map $t(e):=o(\ol{e})$. A graph \textbf{automorphism} is a permutation of $V\Gamma$ and $E\Gamma$ which respects the maps $o$ and $r$. For a graph $Y$ and a vertex $w\in VY$, we define $E_Y(w)$ to be the edges with origin $w$. When clear from context, we suppress the subscript $Y$. We say a graph $Y$ is \textbf{$d$-regular} if $|E_Y(w)|=d$ for all $w\in VY$.

For graphs $X$ and $Y$, a \textbf{graph homomorphism} $\phi:X\rightarrow Y$ is given by two functions $\phi_V:VX\rightarrow VY$ and $\phi_E:EX\rightarrow EY$ such that $\phi_V\circ o =o\circ \phi_E$ and $\phi_E(\ol{e})=\ol{\phi_E(e)}$. We say a graph $X$ is a \textbf{covering graph} of a graph $Y$ if there is a graph homomorphism $\phi:X\rightarrow Y$ such that $\phi_V$ is surjective and $(\phi_E)\rest_{E_X(v)}:E_X(v)\rightarrow E_Y(\phi_V(v))$ is a bijection for all $v\in VX$. We call the homomorphism $\phi$ a \textbf{covering map}. The automorphism group $\Aut(X)$ of the graph $X$ is endowed with the topology of pointwise convergence for its natural action on the set $VX\sqcup EX$, viewed as a discrete set. In particular, if $X$ is connected and locally finite, then $\Aut(X)$ is a second countable totally disconnected locally compact group.

\begin{thm}[{\cite[Section 5.4]{Serre80}}]\label{thm:graph_cover}
Let $X$ be a connected graph and   $H\leq \Aut(X)$ be a closed subgroup acting  without edge inversion. The following hold.
\begin{enumerate}
\item There is a covering map $\phi:T \rightarrow X$, where $T$ is the a tree. If $X$ is additionally $n$-regular, then $T$ is $n$-regular.
\item There is a closed subgroup $G\leq \Aut(T)$ and a continuous surjective homomorphism $\Phi \colon G\rightarrow H$ with the following properties:
\begin{enumerate}[(a)]
\item $\mathrm{Ker}(\Phi)$ is discrete.  In particular, if $H$ is unimodular, then so is $G$. 
\item The following diagrams commute for all $g\in G$:
\[
\xymatrix{
VT \ar[d]^{\phi_V} \ar[r]^g & VT \ar[d]^{\phi_V}\\
VX \ar[r]^{\Phi(g)} & VX} \hspace{1cm}
\xymatrix{
ET\ar[d]^{\phi_E} \ar[r]^g & ET \ar[d]^{\phi_E}\\
EX \ar[r]^{\Phi(g)} & EX.}
\]
\end{enumerate}
\end{enumerate}
\end{thm}

We call the map $\Phi$ given by Theorem~\ref{thm:graph_cover} the \textbf{covering homomorphism}. The group $G$ is called the \textbf{lift} of $H$ to $T$.

\medskip
Throughout this paper, given a group $G$ acting on a set $X$,   the pointwise fixator of $Y\subset X$ is denoted by $G_{(Y)}$. When $X$ is a graph, we set 
$$G^+:=\grp{G_{(e)}\mid e\in ET}.$$
We record the following fact from Bass--Serre theory.

\begin{prop}\label{prop:graph of groups} 
	Let $T$ be a tree and $H$ be a locally compact group acting continuously on $T$  with compact edge stabilizers. The quotient graph $T/H^+$ is a tree. 
\end{prop}
\begin{proof}
The group $H^+$ acts without inversion on $T$, since it is generated by edge stabilizers. Bass--Serre theory ensures that $H^+$ is isomorphic to the fundamental group of a graph of groups whose underlying graph is $X:=T/H^+$; see \cite[Theorem 13]{Serre80}. Additionally, the fundamental group of a graph of groups maps onto the fundamental group of the underlying graph; see \cite[Section 5.1]{Serre80}.  

Suppose toward a contradiction that $X$ admits a cycle. The fundamental group of $X$ thus maps onto $\Zbb$, hence $\Zbb$ is a quotient of $H^+$. Any homomorphism of a locally compact group to $\Zbb$ is continuous via \cite[Corollary 3]{Alperin}. Any homomorphism $H^+ \to \Zbb$ therefore has a trivial image, since $H^+$ is generated by compact subgroups and $\Zbb$ has no non-trivial finite subgroups. This is absurd. We conclude that $X$ is indeed a tree. 
\end{proof}

\subsection{Normal subgroups of groups acting on trees}

The following basic fact seems to be due to J.~Tits. The reader may consult \cite[Lemma 4.2]{LeBou16} for a proof of the first two claims. The claim on non-amenability follows easily from the existence of discrete free subgroups afforded by a standard Ping-Pong argument; a detailed proof may be found in \cite[Theorem 1]{Neb88}.

\begin{prop}[Tits]\label{prop:Tits}
	Let $T$ be a  tree with more than two ends. If $G \leq \Aut(T)$ acts minimally without a fixed end, then every non-trivial normal subgroup $N$ of $G$ is such that it acts minimally without a fixed end, has a trivial centralizer in $\Aut(T)$, and is not amenable. 
\end{prop}	

There is an important normal subgroup of $G\leq \Aut(T_n)$, where $T_n$ is the $n$-regular tree. Define $\sim$ on $VT_n$ by $v\sim w$ if and only if $d(v,w)$ is even. This is a $G$-equivariant equivalence relation on $VT_n$ which partitions $VT_n$ into two parts. The subgroup $G^*$, of index two in $G$, is the collection of $g$ that do not interchange the parts. The subgroup $G^*$ acts on $T_n$ without edge inversion. Conversely, any subgroup of $G$ which acts on $T_n$ without inversion is contained in $G^*$. 

\subsection{Burger--Mozes groups}

Let $Y$ be a graph and $d> 0$ be an integer. A \textbf{coloring of degree $d$} of  $Y$ is a map $ c \colon EY\rightarrow [d]$ such that  for every $v\in VY$, the restriction
\[
c_v:=c\rest_{E_Y(v)}:E_Y(v)\rightarrow [d]
\]
is either a bijection or is constant. A vertex $v$ is called \textbf{$c$-regular} or \textbf{$c$-singular} accordingly. The coloring $c$ is called \textbf{regular} if all vertices are $c$-regular. In that case $Y$ is $d$-regular. 

Let $g\in \Aut(Y)$, $c$ be a coloring of degree $d$, and $v \in VY$ be a vertex such that $v$ and $gv$ are both $c$-regular. We may then define the \textbf{local action} of $g$ at $v$ as the permutation of $[d]$ given by
\[
\sigma_c(g,v):=c_{g(v)}\circ g\circ c^{-1}_v.
\]
When clear from context or unimportant, we suppress the subscript $c$. The local action enjoys two important properties, the proofs of which are easy exercises:
\begin{equation}
	\sigma(gh,v)=\sigma(g,hv)\sigma(h,v),
\end{equation}
and
\begin{equation}
	\sigma(g^{-1},v)=\sigma(g,g^{-1}v)^{-1}
\end{equation}
for all automorphisms $g, h$ preserving the set of $c$-regular vertices. 

\begin{defn}
Let $d>2$ and let $T$ be the $d$-regular tree. For a permutation group $F\leq\Sym(d)$ and $c \colon ET\rightarrow [d]$ a regular coloring, the \textbf{Burger--Mozes group} is
\[
U_c(F):=\{g\in \Aut(T)\mid \forall v\in VT\;\sigma_c(g,v)\in F\}.
\]
We write $U_c((F,[d]))$ when we wish to emphasize the permutation representation of $F$.
\end{defn}

The group $U_c(F)$ depends on the coloring $c$, and this dependence is somewhat mysterious. For instance, it is easy to construct regular colorings such that $U_c(F)$ is finite. There is a class of colorings, however, for which we have good control over the resulting group.

\begin{defn}
A coloring $c$ of a graph $Y$ is \textbf{legal} if $c(e)=c(\ol{e})$ for each edge $e\in EY$. 
\end{defn}

\begin{prop} \label{prop:BuMo}
	Let $F, F' \leq \Sym(d)$ and let $c, c'$ be regular legal colorings of the $d$-regular tree $T$. 
	If $F$ and $F'$ are isomorphic as permutation groups, then  $U_c(F)$ is conjugate to $U_{c'}(F')$ by some $g\in \Aut(T)$. In particular, the isomorphism type of $U_c(F)$ is independent of the choice   of regular legal coloring.
\end{prop}
\begin{proof}
In case $F = F'$, the required assertion is proved by Burger and Mozes in  \cite[Section 3.2]{BuMo}.
Assume now that $F$ and $F'$ are distinct and isomorphic. There thus exists $h \in \Sym(d)$ with $hFh^{-1} = F'$. Let $d$ be the coloring defined by $d(e):=hc(e)$. Clearly, $d$ is again a regular legal coloring. For $g\in U_c(F)$ and $v\in VT$, we see that
\[
\sigma_{d}(g,v)=d_{g(v)}\circ g\circ d_v^{-1}=hc_{g(v)}\circ g\circ c_v^{-1}h^{-1}=h\sigma_c(g,v)h^{-1}.
\]
We conclude that $\sigma_{d}(g,v)\in F'$ for all $g\in G$ and $v\in VT$. Hence, $U_c(F)\leq U_{d}(F')$. The converse inclusion is similar, and thus $U_c(F)=U_d(F')$. We conclude from our initial observation that $U_c(F)$ is indeed conjugate to $U_{c'}(F')$. 
\end{proof}

For $c$ a regular legal coloring, one easily verifies that $U_c(\{1\})$ acts vertex transitively on $T$. Therefore, $U_c(F)$ acts vertex transitively on $T$ for any regular legal coloring $c$ and permutation group $F$. 

\begin{defn} 
For $d\geq 3$, $F\leq \Sym(d)$, and $c$ some (equivalently, any) regular legal coloring, we call the group $U_c(F)$ the \textbf{Burger--Mozes universal group} with local action prescribed by $F$ and denote it by $U(F)$.
\end{defn}

The term `universal' is justified by \cite[Proposition~3.2.2]{BuMo}, where it is shown  that if $F \leq \Sym(d)$ is transitive, then every vertex-transitive subgroup $H \leq \Aut(T_d)$ whose local action at some vertex is isomorphic to $F$ is contained in $U_c(F)$ for some regular legal coloring $c$. Adapting the argument loc. cit., we obtain a slightly more general fact, which covers the case where $F$ is intransitive. 

\begin{prop}[{cf. \cite[Proposition~3.2.2]{BuMo}}]\label{prop:Universality_BuMo}
	Let $T$ be a tree and $H \leq \Aut(T)$. Suppose that we have a bijection $k \colon E(v) \to [d]$ for some $v\in VT$ and $d>2$ and set
	\[
	F := \{\sigma_k(g, v) \mid g \in H_{(v)}\} \leq 	\Sym(d).
	\] 
	The following assertions hold. 
	
	\begin{enumerate}[(i)]
		\item There exists a  coloring $\tilde k$ of $T$ satisfying the following properties:
		
		\begin{enumerate}[(a)]
			\item  $\tilde k \rest_{E(v)} = k$, 
			\item the set of $\tilde{k}$-regular vertices coincides with the $H$-orbit of $v$,
			\item   $\sigma_{\tilde k}(g, w) \in F$ for all $g \in H$ and $w$ in the $H$-orbit of $v$.
		\end{enumerate}   
		 
		\item If $H$ is vertex-transitive, then   $H \leq U_{\tilde k}(F)$ and $\tilde{k}$ is regular coloring. 
		
		\item In addition, if either $F$ is transitive or every edge of $T$ is inverted by some element of $H$, then $\tilde k$ can be chosen to be in addition legal. 
	\end{enumerate}
\end{prop}
\begin{proof}
Let $H.v$ denote the orbit of $v$ under the action of $H$. For each $w\in H.v$, fix $h_w\in H$ such that $h_w(w)=v$. For $w=v$, let us take $h_v=1$. We now define $\tilde{k}:ET\rightarrow [d]$ by
\[
\tilde{k}(e):=
\begin{cases}
0 & o(e)\notin H.v\\
k(h_{o(e)}(e)) & o(e)\in H.v
\end{cases}.
\]
That $\tilde{k}$ is a coloring, claim (a), and claim (b) are obvious from the definition of $\tilde{k}$. 

For claim (c), take $h_w$ is as fixed in the definition of $\tilde{k}$. We now compute 
\[
\sigma_{\tilde{k}}(h_w,w)=\tilde{k}_v\circ h_w\circ \tilde{k}_w^{-1}=k\circ h_w\circ (k\circ h_w)^{-1}=1.
\]
We deduce that $\sigma_{\tilde{k}}(h_w,w)=1$ and that $\sigma_{\tilde{k}}(h_w^{-1},v)=1$. For an arbitrary $g\in H$ and $w\in H.v$, the element $h_{g(w)}gh_w^{-1}$ fixes $v$, so by definition of $F$,
\[
\sigma_{\tilde{k}}(h_{g(w)}gh_w^{-1},v)\in F.
\]
On the other hand,
\[
\begin{array}{rcl}
\sigma_{\tilde{k}}(h_{g(w)}gh_w^{-1},v) & = &\sigma_{\tilde{k}}(h_{g(w)},gh_w^{-1}(v))\sigma_{\tilde{k}}(gh_w^{-1},v)\\
										& = & \sigma_{\tilde{k}}(h_{g(w)},g(w))\sigma_{\tilde{k}}(g,h_w^{-1}(v))\sigma_{\tilde{k}}(h_w^{-1},v)\\
										& = & 1\cdot \sigma_{\tilde{k}}(g,w)\cdot 1\\
										& = & \sigma_{\tilde{k}}(g,w).
\end{array}
\]
We deduce that $\sigma_{\tilde{k}}(g,w)\in F$ for all $g\in H$ and $w\in H.v$.

\medskip

Claim (ii) is immediate from claim (1c). 

\medskip

For Claim (iii), let $\tilde{k}$ be the coloring given by claim (i). For all  $i,j\in [d]$, if there is some element of $F$ that carries $i$ to $j$, then fix $g_{ij}\in F$ such that $g_{ij}(i)=j$. We assume that $g_{ii}=1$. Observe that we have such a $g_{ij}$ whenever $i=\tilde{k}(e)$ and $j=\tilde{k}(\ol{e})$ for some edge $e$, by our hypotheses. 

Fix $w\in VT$ and for each vertex $v\in VT\setminus\{w\}$, let $e_v$ be the edge with origin $v$ on the geodesic from $v$ to $w$. We now define a legal coloring $c:ET\rightarrow [d]$ on $E(v)$ by induction on $d(v,w)$ such that for each $e\in ET$ there is $g\in F$ such that  $c(e)=g\tilde{k}(e)$. For the base case, we set $c(f):=\tilde{k}(f)$ for all $f\in E(w)$. Suppose that we have defined $c$ on $E(v)$ for all $v\in B_n(w)$. Take $v$ such that $d(v,w)=n+1$. Since $t(e_v)\in B_n(w)$, $c$ is defined on $\ol{e_v}$. Say that $c(\ol{e_v})=g\tilde{k}(\ol{e}_v)=g(j)$ and $\tilde{k}(e_v)=i$. We set $c(f):=gg_{ij}\tilde{k}(f)$ for $f\in E(w)$. It follows that $c$ is a legal coloring.  

It is clear that the coloring $c$ satisfies (a) and (b) of claim (i). Let us argue for (c). Taking $g\in H$ and $w\in H.v$,
\[
\sigma_c(g,w)=c_{g(w)}\circ g\circ c_w=z\tilde{k}_{g(w)}\circ g\circ \tilde{k}_w^{-1}y=z\sigma_{\tilde{k}}(g,w)y
\]
for some $z,y\in F$. Since $\sigma_{\tilde{k}}(g,w)\in F$ by claim (i), we infer that $\sigma_c(g,w)\in F$, verifying (c).
\end{proof}

We finish this subsection with some supplementary results that will be useful in recognizing when a subgroup $H \leq \Aut(T)$ is conjugate to a subgroup of $U(F)$. 

\begin{lem}\label{lem:LocalAction-UniqueFixedPoint}
Let $T$ be a tree and $H \leq \Aut(T)$ be a closed subgroup. Let $v \in VT$ be such that the action of $H_{(v)}$ on $E(v)$ has a unique fixed point $e$. If $H$ is unimodular, then for every $h \in H$ with $hv = t(e)$, we have   $he = \bar e$. 
\end{lem}
\begin{proof}
We prove the contrapositive. Suppose that $he \neq \bar e$. The subgroup $H_{(e)}$   fixes $hv$, and hence it also fixes $he$. We deduce that $H_e \leq H_{he} = h H_e h^{-1}$. On the other hand the unique fixed point of $H_{hv}$ on $E(hv)$ is $he \neq \bar e$, so  $H_e \neq H_{he}$. The compact open subgroup $H_{he}$ is thus conjugate to a proper subgroup of itself, preventing $H$ from being unimodular.
\end{proof}

\begin{cor}\label{cor:LocalAction-fp}
	Let $d \geq 2$ and let $F \leq \Sym(d+1)$ be a permutation group fixing $0$ and acting transitively on $\{1, \dots, d\}$. Let $H \leq \Aut(T_{d+1})$ be a  vertex-transitive, unimodular, closed subgroup whose local action is isomorphic to $F$. Then there is a regular legal coloring $c$ of $T_{d+1}$ such that $H$ is a subgroup of $U_c(F)$. 
\end{cor}

\begin{proof}
Let $v \in VT$ and $e \in E(v)$. Since $H$ is vertex-transitive, there exists $h \in H$ with $hv=t(e)$. If $e$ is the unique fixed point of $H_{(v)}$ on $E(v)$, then $h(e) = \bar e$ by Lemma~\ref{lem:LocalAction-UniqueFixedPoint}. Otherwise, let $e' \neq e$ be the unique fixed point of $H_{(v)}$ on $E(v)$. Thus $he' \neq he$ is the unique fixed point of $H_{(hv)}$ on $E(hv)$. By Lemma~\ref{lem:LocalAction-UniqueFixedPoint}, we have $he' \neq \bar e$, so $he$ and $\bar e$ lie in the same $H_{(hv)}$-orbit on $E(hv)$. There is thus $g \in H_{hv}$ with  $ghe = \bar e$. 

In either case, we have shown that the edge $e \in E(v)$ can be inverted. Since $e$ was arbitrary, we conclude from Proposition~\ref{prop:Universality_BuMo}(iii) that $H \leq U_c(F)$ for some regular legal coloring $c$. 
\end{proof}

\subsection{The group $G^+$}

Given a tree $T$ and a subgroup $G\leq \Aut(T)$, recall that $G^+ =\grp{G_{(e)}\mid e\in ET}$. The subgroup $G^+$ is normal, and if $G$ is closed in $\Aut(T)$, then $G^+$ is open (hence closed) in $G$. In particular, $G$ is discrete if and only if $G^+$ is discrete. The group $G^+$ plays an important role in the setting of the groups $U(F)$.

\begin{prop}[{\cite[Proposition~3.2.1]{BuMo}}]\label{prop:BuMoSimple}Let $F\leq \Sym(d)$ with $d>2$. 
	\begin{enumerate}[(i)]
		\item $U(F)$ is discrete if and only if $F$ acts freely if and only if $U(F)^+$ is trivial. 
		\item If $F$ does not act freely, then $U(F)^+$ is abstractly simple. 
		
		\item $[U(F) :U(F)^+]$ is finite if and only if $[U(F) :U(F)^+]=2$ if and only if $F$ is transitive and generated by its point stabilizers. 
	\end{enumerate}
\end{prop}

It is important to notice that  $U(F)^+$ need not be compactly generated. Indeed, we have the following (see Definition~\ref{def:Ss} for the definition of $\ms S$). 

\begin{cor}\label{cor:U(F)^+}
	$U(F)^+ \in \ms S$ if and only if  $F$ is transitive and generated by its point stabilizers. 
\end{cor}
\begin{proof}
	The `if' part follows from Proposition~\ref{prop:BuMoSimple}. For the converse, observe that $U(F)^+$ is a non-trivial normal subgroup of the vertex-transitive group $U(F)$, so $ U(F)^+$ does not preserve any non-empty proper subtree by Proposition~\ref{prop:Tits}. Since $U(F)^+$ is compactly generated, it follows from \cite[Lemma~2.4]{CaDM} that $U(F)^+$ acts cocompactly on $T$. Therefore, $U(F)/U(F)^+$ is compact, hence finite since  $U(F)^+$ is open. The group  $F$ is thus transitive and generated by its point stabilizers by Proposition~\ref{prop:BuMoSimple}.
\end{proof}

\section{Virtual indicability}

\begin{thm}\label{thm:Indic}
	Let $G$ be a topological group with a  continuous action  by automorphisms on an infinite  locally finite leafless tree $T$. We assume that $G$ has finitely many orbits of vertices, and that for every vertex $v \in VT$,  the local action $F(v)\leq \Sym(E(v))$ of $G$ at  $v$ is such that the subgroup of $F(v)$ generated by its point stabilizers is intransitive on $E(v)$.  Then $G$ is virtually indicable.
\end{thm}

\begin{proof}
	Let $\varphi \colon G\rightarrow \Aut(T)$ be the induced homomorphism and set $H: = \overline{\varphi(G)}$. The group $H^+$ is open in $H$, so the restriction of the quotient map $H\rightarrow H/H^+$ to the dense subgroup $\varphi(G)$ is surjective. It thus suffices to show that $H/H^+$ is virtually indicable. 
	 
	 If $H$ does not contain any hyperbolic element, then by \cite[Proposition~3.4]{Tits}, either $H$ fixes a vertex or inverts an edge, or $H$ fixes an end and preserves each horoball centered at that end. In either case, we get a contradiction with the hypotheses that $T$ is infinite leafless and that the $G$-action has finitely many orbits of vertices. Thus $H$ contains hyperbolic elements.  
	 
	Assume next that $H$ fixes an end $\xi \in \partial T$. Since  $H$ contains a hyperbolic element, it permutes the horoballs centered at $\xi$ non-trivially. The Busemann homomorphism\footnote{Fixing a representative ray $x_0,x_1,\dots$ of the end $\xi$, the Busemann homomorphism is the map $H\rightarrow \mathbf{Z}$ defined by $g\mapsto \lim_{i\rightarrow \infty} (d(g(x_i),x_0)-i-1)$.} associated to $\xi$ yields a continuous, surjective homomorphism $H \to \mathbf Z$ vanishing on $H^+$. The group $H/H^+$ is thus virtually indicable,  as desired. 
	
	We assume henceforth that $H$ does not fix an end. The quotient graph $X:=T/H^+$ is a tree by Proposition~\ref{prop:graph of groups}, and the natural action of the discrete quotient group   $H/H^+$ on  $X$ is proper and cocompact.  We shall argue that $X$ is an infinite tree by showing that each vertex of $X$ has degree at least~$2$. It then follows that $H/H^+$ is virtually an infinite free group and is thus virtually indicable, as required. 
	
	 Fix a vertex $w_0 \in VT$, let $W$ be the $H$-orbit of $w_0$ in $VT$, and say that $m$ is the degree of $w_0$ in $T$. Our assumption on the local action of $G$ ensures that $m\geq 2$. Fix $c_0 \colon E(w_0) \to [m]$ a bijection and set $F:= \{ \sigma_{c_0}(g, w_0) \mid g \in H_{(w_0)}\}$. The number of orbits of $  F$ on $[m]$ is a lower bound for the degree of the image of $w_0$ in the quotient graph $X$. Therefore, if $ F$ is not transitive on $[m]$, then the image of  $w_0$ in the quotient graph $X$ has degree at least~$2$. Let us assume that $F$ is transitive on $[m]$. Proposition~\ref{prop:Universality_BuMo} provides  a coloring   $c \colon ET \to [m]$ extending $c_0$ such that for each $w \in W$, the restriction $c\rest_{E(w)} \colon E(w) \to [m]$ is bijective,  for each $y\in VT \setminus W$, the  restriction $c \rest_{E(y)} \colon E(y) \to [m]$ is constant, and $ \sigma_c(g, w) \in F$ for all $g \in H$ and $w \in W$. As $F$ acts transitively, Proposition~\ref{prop:Universality_BuMo} ensures further that $k$ can be chosen in such a way that $k(e)=k(\bar e)$ for all $e \in ET$. 
	
	Let $F^+$ be the subgroup of $F$ generated by the point stabilizers. Recall that by hypothesis, $F^+$ is intransitive and let  $B_1, \dots, B_p$ list the $F^+$-orbits on $[m]$. The $F^+$-orbits form an $F$-equivariant equivalence relation on $[m]$, and as is customary,  we call the $B_i$ \textit{blocks}.  Let  $\pi \colon   F \to \Sym(\{B_1, \dots , B_p\})$ be the $F$-action on the blocks. It is easy to see that $\Ker(\pi)= F^+$ and that the $F/F^+$-action on the blocks is free. 
	
	Fix $g \in H$ and suppose that $w,w'\in W$ are such that every vertex different from $w$ and $w'$ on the geodesic $[w,w']$ is not in $W$. We infer that $c(e) = c(f)$ for all edges $e, f$ on  $[w, w']$, since $c(e)=c(\ol{e})$ for all edges $e$. The pair $gw$ and $gw'$ also enjoys the same condition on $[gw,gw']$, so $c(e) = c(f)$ for all edges $e, f$ on $[gw, gw']$. If $\sigma_c(g, w)\in \Ker(\pi)$, then the elements $c(e)$ and $c(ge)$ belong to the same block for all $e \in E(w)$. Therefore, $c(e)$ and $c(ge)$ belong to the same block for all edges $e$ on $[w, w']$. In particular, the edge $f \in E(w')$ on the geodesic from $w'$ to $w$ is such that $c(f)$ and $c(gf)$ belong the same block. Thus, $\sigma_c(g, w')$ fixes a block, and hence $\sigma_c(g,w')$ belongs to $\Ker(\pi)$, since the $F/F^+$-action on the blocks is free. The obvious induction argument on $d(w,w')$ now shows that for all $g\in H$ and  $w,w'\in W$, $\sigma_c(g,w)\in \Ker(\pi)$ if and only if $\sigma_c(g,w')\in \Ker(\pi)$. 
	
	Take $g \in H$ which fixes  pointwise an edge $e$ in $T$. The local action $\sigma_c(g,o(e))$ must fix a block, so $\sigma_c(g,o(e))\in \Ker(\pi)$, since the action of $F/F^+$ on the blocks is free. By the previous paragraph, $\sigma_c(g,w) \in \Ker(\pi)$ for all $w\in W$. As $H^+$ is generated by edge fixators, it follows that $\sigma_c(g,w)\in \Ker(\pi)$ for every $g\in H^+$ and $w\in W$. The degree of the image of $w_0$ in $X$ is  thus at least the number of blocks, which is at least $2$.  
	
	We conclude that all vertices of the tree $X$ have  degree at least two, completing the proof.
\end{proof}

\begin{rmk}
The hypothesis that $T$ be locally finite in Theorem~\ref{thm:Indic} is essential. Indeed, consider the action of the free product $G = S_1 \ast S_2$ of two infinite simple groups $S_1, S_2$ on its Bass--Tree $T$, which is locally infinite since $S_1$ and $S_2$ are infinite. The $G$-action on $T$  enjoys all the properties  required by Theorem~\ref{thm:Indic}, but $G$ is not virtually indicable: indeed $G$ is perfect and does not have any proper subgroup of finite index. 
\end{rmk}

A class of finite permutation group satisfying the condition in the theorem is provided by the nilpotent groups. 

\begin{lem}\label{lem:NilpPerm}
	Let $F \leq \Sym(\Omega)$ be a permutation group on a finite set $\Omega$ of size at least $2$. If $F$ is nilpotent, then the subgroup of $F$ generated by the point stabilizers is intransitive. 
\end{lem}
\begin{proof}
As the result is clear for $F$ intransitive, we may assume that $F$ is transitive. There exists an   $F$-invariant equivalence relation $\sim$ on $\Omega$ such that the $F$-action on $\Omega/\sim$ is primitive and non-trivial. The only nilpotent primitive permutation groups are cyclic of prime order. Letting $F^+$ be the subgroup of $F$ generated by the point stabilizers, the $F$-action on the $\sim$-equivalence classes is through a quotient that acts freely. We conclude that the $F^+$-action  on the $\sim$-equivalence classes is trivial. In particular, every $F^+$-orbit is entirely contained in some  $\sim$-equivalence class, so $F^+$ is intransitive. 
\end{proof}

\begin{proof}[Proof of Corollary~\ref{cor:nilpotentLocal}]
We assume that $G$ fixes neither a vertex nor an edge. Since $G$ is compactly generated, it follows that $G$ contains a hyperbolic element. If $G$ fixes an end $\xi$ of $T$, the Busemann homomorphism at $\xi$ yields an infinite cyclic quotient of $G$, so $G$ is virtually indicable. We thus assume that $G$ also does not fix any end. There exists a minimal non-empty $G$-invariant subtree $X$ in $T$.  Since the $G$-action is fixed-point-free, the tree $X$ is infinite, and $X$ is locally finite since $T$ is locally finite. The local action of the image of $G$ in $\Aut(X)$ at every vertex $v$ of $X$ is a quotient of the local action of $G \leq \Aut(T)$ at $v$. The group $G$ thus acts on $X$ with a nilpotent local action at every vertex. By Lemma~\ref{lem:NilpPerm}, all hypotheses of Theorem~\ref{thm:Indic} are satisfied, and the conclusion follows.
\end{proof}

\section{Simple subquotients}
In this section, we prove Theorem~\ref{thmintro:Subquotient}. That is, we show every non-discrete Burger--Mozes group admits a subquotient belonging to the class $\ms{S}$.

\subsection{Reduction to cyclic groups of order $p$ acting on $p+1$ points}

The first step in the proof of Theorem~\ref{thmintro:Subquotient} consists in reducing the problem from all non-free permutation groups $F$ to the rather odd class of permutation groups $(C_p, [p+1])$ of cyclic groups of prime order $p$ acting on $p+1$ points. Notice that the cyclic group $C_p$ has only one faithful permutation representation on $[p+1]$ up to conjugacy.  Thus, by   Proposition~\ref{prop:BuMo}, the group $U_c((C_p,[p+1]))$ is uniquely defined and does not depend on the choice of the regular legal coloring $c$. We shall denote it by  $U(C_p)$. 

The required reduction step is realized by the following observation. 

\begin{lem}\label{lem:PrimeReduction}
Let $F\leq \Sym(d)$ be a permutation group and $c$ be a regular legal coloring of the $d$-regular tree $T$. If $F$ does not act freely, then  $U_c(F)$ contains a closed subgroup isomorphic to $U(C_p)$ for some prime $p$.
\end{lem}
\begin{proof}
Let $i\in [d]$ be such that $F_{(i)}$ is non-trivial. Replacing the coloring by $\sigma \circ c$ for a suitable permutation $\sigma$ of $[d]$,   we may assume that $i=0$.   Let $p$ be a prime dividing the order of $ F_{(0)}$ and let $C_p\leq F_{(0)}$ be a non-trivial cyclic subgroup of  order $p$. We may find a set  $\Omega\subseteq [d]$ of size $p+1$ such that $0\in \Omega$ and $C_p$ cyclically permutes $\Omega\setminus \{a\}$. By changing the coloring again if needed, we may assume that $\Omega=[p+1]$.

Fix a vertex $v\in VT$ and let $X$ be the subtree spanned by the edges colored by $[p+1]$. The tree $X$ is plainly a copy of $T_{p+1}$, and the restriction of the ambient coloring to $X$ gives an action of $U(C_p)$ on $X$. Taking $g\in U(C_p)$, we may extend the action of $g$ on $X$ to the whole tree $T$ by declaring the local action to be trivial on $VT \setminus VX$. We thus deduce that $U(C_p)\leq U(F)$.	
\end{proof}

The following subsidiary fact will be useful to identify $U(C_p)$ in a context where   certain a priori illegal colorings are allowed.

\begin{lem}\label{lem:legal_Cp}
	Let $F\leq \Sym(p+1)$ be a permutation group which fixes $0$ and transitively permutes $\{1,\dots,p\}$ and let $T$ be the $(p+1)$-regular tree. If $c$ is a regular coloring such that $c(e)=0$ implies $c(\ol{e})=0$, then $U_c(F)=U_d(F)$ where $d$ is a regular legal coloring. In particular, if $p$ is a prime we have $U_c(C_p)=U(C_p)$.
\end{lem}
\begin{proof}
	Fix a vertex $v\in VT$ and for each vertex $w\neq v$, define $e_w$ to be the edge with origin $w$ on the geodesic from $w$ to $v$.  For each $i, j\in \{1,\dots, p\}$, fix $g_{ij}\in F$ such that $g_{ij}(i)=j$; for $i=j$, we take the element $g_{ij}$ to be trivial. 
	
	We define the coloring $d$ by defining the value of $d$ on each $e\in E(w)$ by induction on $d(v,w)$. For the base case, $w=v$, we put $d(e)=c(e)$ for each $e\in E(v)$. Suppose we have defined $d$ on $E(w)$ for all $w\in B_n(v)$. Take $w$ such that $d(w,v)=n+1$. If $c(\ol{e}_w)=0$, then we set $d(e)=c(e)$ for all $e\in E(w)$. If $c(\ol{e_w})=j\neq 0$, then we set $d(e):=g_{kj}c(e)$ for $e\in E(w)$ where $k:=c(e_w)$. The function $d$ is clearly a legal coloring.  
	
	Taking $g\in U_c(F)$, let us compute the local action of $g$ according to $d$:
	\[
	\sigma_d(g,v)=d_{g(v)}\circ g\circ d_v=zc_{g(v)}\circ g\circ c_v^{-1}y=z\sigma_c(g,v)y
	\]
	for some $z,y\in F$. Since $g\in U_c(F)$, we infer that $\sigma_c(g,v)\in F$, hence $\sigma_d(g,v)\in F$. We conclude that $U_c(F)\leq U_d(F)$. The converse inclusion is similar.
\end{proof}

\subsection{Frobenius groups}
In view of Corollary~\ref{cor:U(F)^+}, the group $U(C_p)^+$ is not a member of $\ms S$. In fact, by Theorem~\ref{thm:Indic}, we know that it is virtually indicable. In order to show that $U(C_p)$ has a simple subquotient in $\ms S$, we shall show that $U(C_p)$ has a simple subquotient of the form $U(F)^+$, where $F$ is a Frobenius group of a specific kind that will be associated to $p$. Let us first recall the Frobenius groups.
\begin{defn}
A \textbf{Frobenius} group is a transitive permutation group $(F,[n])$ such that the action is not free, but the  stabilizer of every ordered pair of distinct points is trivial. A point stabilizer $F_{(i)}$ is called a \textbf{Frobenius complement}.
\end{defn}

We shall need the existence of certain Frobenius groups. To experts the next theorem is likely obvious, but we sketch a proof. We thank C. Reid for pointing out to us this family of examples.

\begin{thm}\label{thm:forbenius_group}
For each prime $p>2$, there is a finite Frobenius group $F$ such that $C_p$ is the Frobenius complement and $|F:C_p|$ is a power of two. 
\end{thm}
\begin{proof}
Since $p$ is coprime to $2$, we may find a non-trivial irreducible representation $\phi \colon C_p\rightarrow GL_n(\mathbf{F}_2)$, where $\mathbf{F}_2$ is the field with two elements. The representation $\phi$ induces an action $C_p\acts \mathbf{F}_2^n$ which is fixed point free. Consider the semidirect product $F:=\mathbf{F}_2^n\rtimes_\phi C_p$. 

We now argue that the action of $F$ on the set of left cosets $F/C_p$ shows that $F$ is a Frobenius group. Certainly this action is transitive and has non-trivial point stabilizers. Suppose that $ C_p\cap hC_ph^{-1}$ is non-trivial. Since $C_p$ has no proper non-trivial subgroups, we deduce that $h\in N_F(C_p)$. The element $h$ has the form $(a,x)$ where $a\in \mathbf{F}_2^n$ and $x\in C_p$, and since $x\in N_F(C_p)$, the element $a$ is in $N_F(C_p)$. Considering the conjugate $(a,1)(1,x)(a^{-1},1)$, it follows that $\phi(x)$ fixes $a$, and since $\phi(x)$ generates $C_p$,  $C_p$ in fact fixes $a$. We conclude that $a=1$ as $C_p$ acts fixed point freely on $\mathbf{F}_2^n$, so $h\in C_p$. Two point stabilizers are thus trivial. 

The second claim of the proposition is immediate. 
\end{proof}

The next proposition is also likely well-known. We again sketch a proof for completeness.
\begin{prop}\label{prop:Frob_gen_stab}
A finite Forbenius group is generated by its point stabilizers.
\end{prop}
\begin{proof}
Let $F\leq \Sym(n)$ be a Frobenius group and $L \leq F$ be the subgroup generated by the point stabilizers. If $i$ and $j$ are in distinct $L$-orbits, then the stabilizer $L_{(i)}=F_{(i)}$ acts freely on the $L$-orbit of $j$, whose size is thus a multiple of $p=|F_{(0)}|$. But $L_{(j)}$ has exactly one fixed point in the  $L$-orbit of $j$, so the size of that orbit is congruent to $1$ modulo $p$. This contradiction shows that $L$ acts transitively. Thus $F = L F_{(0)}=L$, as required.
\end{proof}

\begin{cor}\label{cor:FrobSimple}
	For any non-trivial finite Frobenius group $F$, the group $U(F)^+$ is an abstractly simple open subgroup of index~$2$ in $U(F)$. In particular,  $U(F)^+$ belongs to $\ms S$.
\end{cor}
\begin{proof}
	This is immediate from Proposition~\ref{prop:Frob_gen_stab} and Corollary~\ref{cor:U(F)^+}.
\end{proof}

\subsection{Colorings and blow-ups}
We now develop machinery to build new graphs from trees. This technique will, in particular, allow us to to change the local action in stages to arrive at a subquotient of $U(C_p)$ which has the form $U(F)^+$ for $F$ a Frobenius group. 

\begin{defn} 
Let $T_n$ be the $n$-regular tree with $n\geq 3$ and $c$ be a regular coloring of $T_n$ by $[n]$. The \textbf{blow-up} of $T_n$ relative to $c$ is the graph $B_c(T_n)$ defined by $VB_c(T_n):=VT_n\times[n]$ and $((v,i),(w,j))\in EB_c(T_n)$ if and only if either $v=w$ and $i\neq j$, or $(v,w)\in ET_n$, $c((v,w))=i$, and $c((w,v))=j$. When unimportant or clear from context, we suppress the subscript $c$. 
\end{defn}
\begin{figure}
    \centering
    \includegraphics[width=0.8\textwidth]{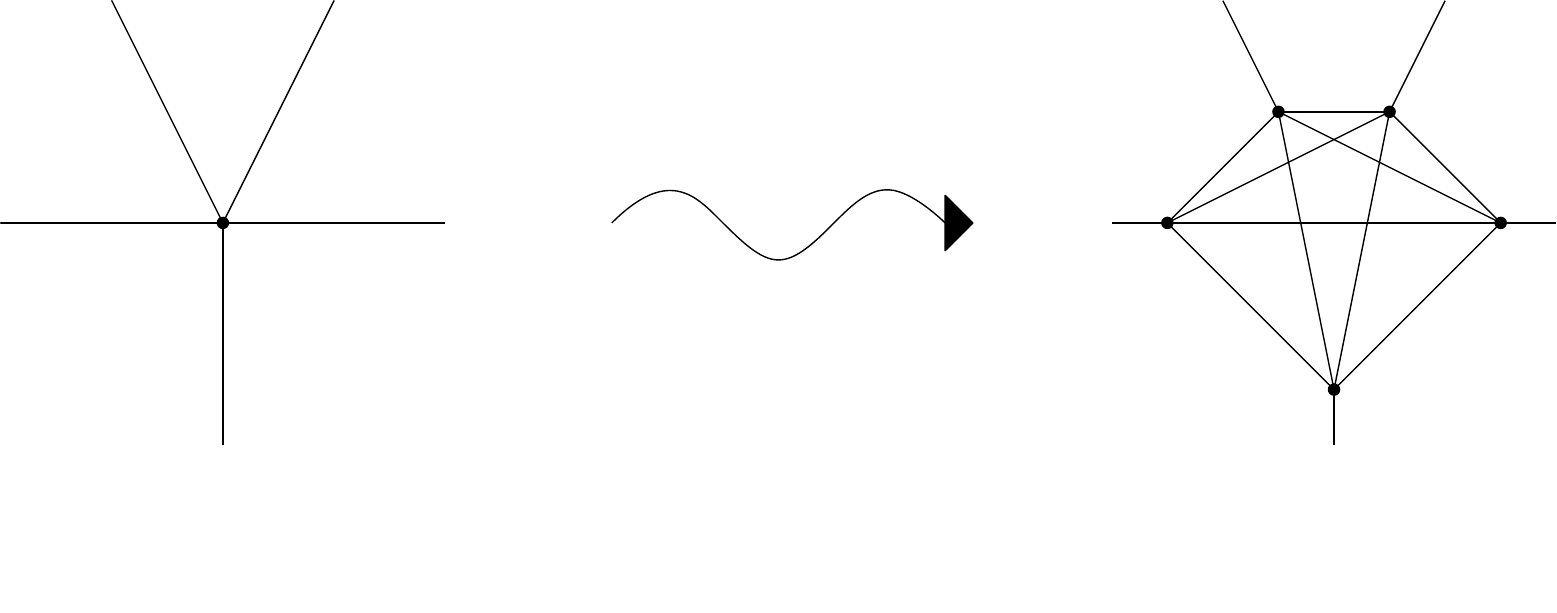}
    \caption{The blow-up of $T_5$.}\label{fig:blowup}
\end{figure}

The blow-up operation replaces the vertices of $T_n$ by complete graphs on $n$ vertices; see Figure~\ref{fig:blowup}. There is also a canonical action of $\Aut(T_n)$ on $B_c(T_n)$ by graph automorphisms:  $g((v,i)):=(g(v),\sigma_c(g,v)(i))$, where $\sigma_c(g,v)$ is the local action relative to $c$. The blow-up of an $n$-regular tree is an $n$-regular graph. We note a useful condition to ensure groups acting on blow-ups do not invert edges.

\begin{lem}\label{lem:no-edge-inversion-BU}
Let $T_n$ with $n\geq 3$, let $c$ be a regular coloring, and suppose that $H\leq \Aut(T_n)$ acts on $T_n$ without edge inversion. If $\grp{\sigma_c(h,v)}$ acting on $[n]$ has no size two orbits for all $v\in VT_n$ and $h\in H_{(v)}$, then $H$ acts on $B_c(T_n)$ without edge inversion.
\end{lem}
\begin{proof}
Fix $h\in H$. Edges in $B_c(T_n)$ have either the form $((v,i),(w,j))$ where $(v,w)\in ET_n$, $c((v,w))=i$, and $c((w,v))=j$ or the form $((v,i),(v,j))$ for some $i\neq j$.

For edges of the form $((v,i),(w,j))$,
\[
h(((v,i),(w,j)))=(h(v),\sigma_c(h,v)(i)),(h(w),\sigma_c(h,w)(j)).
\]
Since $h$ does not invert edges of $T_n$, we see that $h$ cannot invert such an edge.

For edges of the form $((v,i),(v,j))$ for some $i\neq j$,
\[
h(((v,i),(v,j)))=(h(v),\sigma_c(h,v)(i)),h(v),\sigma_c(h,v)(j)).
\]
If $h$ inverts such an edge, then $h(v)=v$, $\sigma_c(h,v)(i)=j$, and $\sigma_c(h,v)(j)=i$, which implies that $\grp{\sigma_c(h,v)}$ has an orbit of size two. We deduce that $h$ also does not invert these edges. The lemma is now verified.
\end{proof}

\begin{defn}
Let $T_n$ be the $n$-regular tree, $c$ be a regular coloring of $T_n$, and $\mc{P}:=\{O_0,\dots, O_{d-1}\}$ be an ordered partition of $[n]$. The \textbf{partition blow-up} of $T_n$ with respect to $c$ and $\mc{P}$ is the graph $B_{c,\mc{P}}(T_n)$ defined as follows: $VB_{c,\mc{P}}(T_n):=VT_n\times \mc{P}$ and $((v,O_i),(w,O_j))\in EB_c(T_n)$ if and only if either $v=w$ and $i-j=\pm 1\mod d$,  or $(v,w)\in ET_n$, $c((v,w))\in O_i$, and $c((w,v))\in O_j$.
\end{defn}

An example of a partition blow-up is depicted in Figure~\ref{fig:Pblowup}. We stress that in the case of an ordered partition into singletons, the partition blow-up does not coincide with the blow-up defined previously. The partition blow-up of $T_4$ with respect to $\mc{P}:=\{\{1\},\{2\},\{3\},\{4\}\}$, for instance, is a three regular graph.

\begin{figure}
    \centering
    \includegraphics[width=0.8\textwidth]{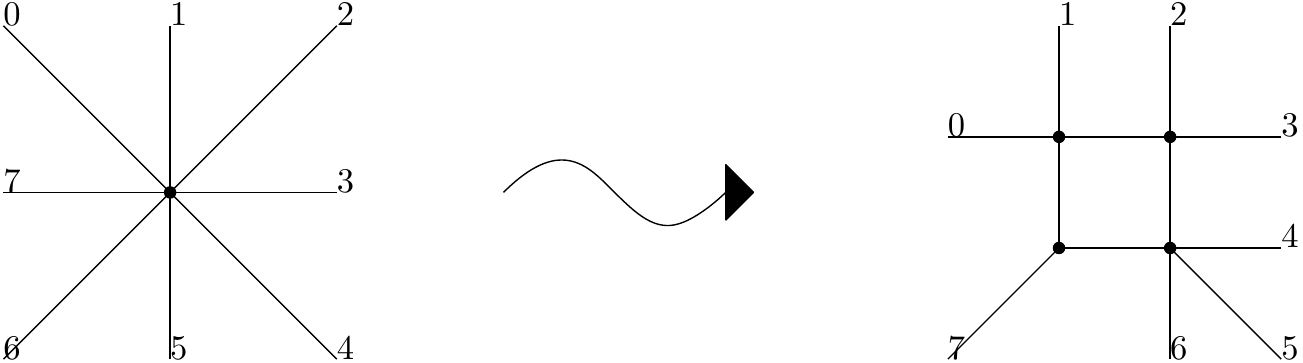}
    \caption{The partition blow-up of $T_8$ relative to the ordered partition $\{\{0,1\},\{2,3\},\{4,5,6\},\{7\}\}$.}\label{fig:Pblowup}
\end{figure}
Given a partition $\mc{Q}$ of $[n]$, the \textbf{Young subgroup} associated to $\mc{Q}$ is the group of permutations of $[n]$ which setwise fix the parts of the partition. Letting $P\leq \Sym([n])$ be the Young subgroup associated to $\mc{P}$, the Burger--Mozes group $U_c(P)$ acts on $B_{c,\mc{P}}(T_n)$ by $g((v,O_i))=(g(v),O_i)$. Taking $U_c^+(P)$, it follows that $U_c^+(P)$ acts on $B_{c,\mc{P}}(T_n)$ without edge inversion. 

We need to consider covering trees of the blowups. Lifting colorings to covering trees thereby becomes important.
\begin{defn}
Let $X$ be an $n$-regular graph with a coloring $c$ and $\phi:T_n\rightarrow X$ the covering map, where $T_n$ is the $n$-regular tree. We call $\tilde{c}:=c\circ \phi$ the \textbf{lifted coloring} on $T_n$ induced by $c$.
\end{defn}

\begin{lem}\label{lem:colorings_covers}
Let $X$ be an $n$-regular graph with a regular coloring $c$, $\phi:T_n\rightarrow X$ be the covering map, and $\Phi:H\rightarrow \Aut(X)$ be the covering homomorphism afforded by Theorem~\ref{thm:graph_cover}. Then the lifted coloring $\tilde{c}$ is a coloring of $T_n$, and
\[
\sigma_{\tilde{c}}(g,v)=\sigma_c(\Phi(g),\phi(v))
\]
for all $g\in H$ and $v\in VT_n$.
\end{lem}
\begin{proof}
That $\tilde{c}$ is a coloring is immediate since $\phi_E$ is a bijection when restricted to $E_{T_n}(v)$.

For the second claim, the following diagram commutes for all $g\in H$ and $v\in VT_n$, via Theorem~\ref{thm:graph_cover}:
\[
\xymatrix{
E_{T_n}(v) \ar[d]^{\phi_E} \ar[r]^g & E_{T_n}(g(v))\ar[d]^{\phi_E}\\
E_X(\phi(v))\ar[r]^{\Phi(g)} & E_X(\phi(g(v)))}.
\]
Defining $\phi_v:=(\phi_E)\rest_ {E(v)}:E_{T_n}(v)\rightarrow E_X(\phi(v))$, we see that 
\[
\begin{array}{rcl}
\sigma_{\tilde{c}}(g,v) & = & c_{\phi(g(v))}\circ \phi_{g(v)}\circ g \circ \phi_v^{-1}\circ c_{\phi(v)}^{-1}\\
				& =  & c_{\phi(g(v))}\circ \Phi(g)\circ c^{-1}_{\phi(v)}\\
				& =  & c_{\Phi(g)(\phi(v))}\circ \Phi(g)\circ c^{-1}_{\phi(v)}\\
				& = & \sigma_{c}(\Phi(g),\phi(v)).
\end{array}
\] 
\end{proof}

\subsection{Simple subquotients}

We are now prepared to prove the main theorem of this section, namely Theorem~\ref{thm:involve_S} below, from which Theorem~\ref{thmintro:Subquotient} will easily follow. The proof of Theorem~\ref{thm:involve_S} makes heavy use of blow-ups, and the argument is inspired by Lemma~9.2 from J. Huang's paper \cite{JH16}.

We begin by proving the $p=2$ case of the main technical theorem; this case requires separate analysis to deal with edge inversions.

\begin{prop}\label{prop:involve_S_p=2}
For $C_2\leq \Sym(3)$, $U(C_2)$ contains a compactly generated closed subgroup $H$ admitting a discrete normal subgroup $D$ such that $H/D$ is isomorphic to $\Aut(T_3)^+$.
\end{prop}

\begin{proof}
Let $G := \Aut(T_3)^+$. Consider the $G$-action on the blow-up $X$ of the trivalent tree $T_3$. This action inverts edges, so we take $\wh{X}$ the barycentric subdivision of $X$. By Theorem~\ref{thm:graph_cover}, the group $G$ lifts to a unimodular closed subgroup $J$ of the automorphism group of the covering tree $T$ of $\wh{X}$. 

The tree $T$ has vertices of degree $2$ and degree $3$; denote the set of degree three vertices by $V_3$. By construction, the local action of $J$ at $v$ for $v\in V_3$ is the cyclic group $C_2$  acting on three points. Define a new tree $T'$ by $VT':=V_3$ and $(v,w)\in ET'$ if and only if either $(v,w)\in ET$ or the geodesic from $v$ to $w$ uses only degree $2$ vertices, other than $v$ and $w$. The graph $T'$ is the three regular tree, $J\leq \Aut(T')$ and acts vertex transitively, and the local action of $J$ at any $v\in T'$ is the cyclic group $C_2$.

 The kernel of the covering homomorphism $\Phi \colon J \to G$ is discrete.  In view of the unimodularity of $J$, it  follows from Corollary~\ref{cor:LocalAction-fp} that $J$ is contained in $U_c(C_2) = U(C_2)$ for some regular legal coloring $c$ of $T_3$.  The group $J$ is thus a compactly generated closed subgroup of $U(C_2)$ that admits $\Aut(T_3)^+$ as a quotient modulo a discrete normal subgroup.
\end{proof}

For the $p>2$ case, we proceed by proving three lemmas.

\begin{lem}\label{lem:1_main_thm}
For $p>2$ a prime, say that $(F,[n])$ is transitive and such that $F_{(0)}=C_p$ and say that $c$ is a regular legal coloring of $T_n$.  Letting $(C_p,[n]):=(F_{(0)},[n])$, there is a regular coloring $d$ of $T_n$ such that
\begin{enumerate}
\item for all $e\in ET_n$, $d(e)=0$ if and only if $d(\ol{e})=0$, and 
\item there is a closed compactly generated subgroup $H\leq U_d((C_p,[n]))^*$ admitting a continuous epimorphism $\Phi:H\rightarrow U_c((F,[n]))^*$ with discrete kernel. 
\end{enumerate}
\end{lem}

\begin{proof}
Set $G:=U_{c}((F,[n]))^*$, let $X$ be the blow-up $B_{c}(T_n)$, and recall that $G$ acts on $X$ by graph automorphisms. Since the vertex stabilizers of $F$ equal $C_p$ for $p>2$, the conditions of Lemma~\ref{lem:no-edge-inversion-BU} are satisfied. We conclude that $G$ acts on $X$ without edge inversion.

For each $i\in [n]$, fix $g_i\in F$ such that $g_i(i)=0$ and say that $g_0=1$; we may find these elements since $(F,[n])$ is transitive. Recalling that each $v\in VX$ has the form $(u,i)$ for some $u\in VT_n$ and $i\in [n]$, we define a coloring $a$ on $X$ by 
\[
a(e):=g_{\pi_2(o(e))}\pi_2(t(e))
\]
where $\pi_2:VT_n\times[n]\rightarrow [n]$ is the projection onto the second coordinate. That $c$ is a regular legal coloring ensures that $a$ is a regular coloring. The coloring $a$ is also such that $a(e)=0$ if and only if $a(\ol{e})=0$, again since $c$ is a legal coloring.

Setting $a_{(v,i)}:=a\rest_{E_X((v,i))}$, the map $a_{(v,i)}^{-1}:[n]\rightarrow E_X((v,i))$ is such that 
\[
a^{-1}_{(v,i)}(j)=((v,i),(u,g^{-1}_i(j))
\]
where $u=v$ if and only if $j\neq 0$. From this observation and a simple, but tedious computation, it follows that
\[
\sigma_{a}(g,(v,i))=g_{\sigma_{c}(g,v)(i)}\cdot \sigma_{c}(g,v)\cdot g^{-1}_i.
\]
The elements $g_j$ are in $F$ for all $j\in [n]$, so $\sigma_{a}(g,(v,i))\in F$. Additionally, $\sigma_a(g,(v,i))(0)=0$. We deduce that $\sigma_{a}(g,(v,i))\in F_{(0)}$ for all $(v,i)\in VX$. Recalling that $F_{(0)}=C_p$, the local action of $G$ on $X$ is $(C_p,[n])=(F_{(0)},[n])$. 

The covering tree of $X$ is $T_n$. Let $d$ be the lift of the coloring $a$ to $T_n$. It follows that $d(e)=0$ if and only if $d(\ol{e})=0$ and that $d$ is a regular coloring.  Applying Theorem~\ref{thm:graph_cover}, we obtain $H\leq \Aut(T_n)$ the covering group of $G$ and $\Phi:H\rightarrow G$ the covering homomorphism. The subgroup $H$ is closed, and the map $\Phi$ has a discrete kernel. Via Lemma~\ref{lem:colorings_covers}, $\sigma_{d}(g,v)$ is an element of $(C_p,[n])$ for all $g\in H$ and $v\in VT_n$, so in fact $H\leq U_{d}((C_p,[n]))$. The group $H$ also cannot invert edges, since $G$ does not, so we conclude that $H\leq U_d((C_p,[n]))^*$. As the group $U_c((F, [n]))$ is compactly generated, we may take $H$ to be compactly generated. The lemma is now verified.
\end{proof}

We now make use of the partition blow-up. 
 
\begin{lem}\label{lem:2_main_thm}
For $p>2$ a prime and $n> 3$ even, say that $(C_p,[n])$ is a permutation group with more that two orbits such that $0$ is the only fixed point and say that $c$ is a regular coloring of $T_n$ such that $c(e)=0$ if and only if $c(\ol{e})=0$. Letting $(C_p,[p+2])$ be the permutation group given by the cycle $(2\dots p+1)$, there is a regular coloring $d$ of $T_{p+2}$ such that
\begin{enumerate}
\item for all $e\in ET_{p+2}$ and $i\in \{0,1\}$, $d(e)=i$ if and only if $d(\ol{e})=i$, and 
\item there is a closed subgroup $H\leq U_d((C_p,[p+2]))^*$ admitting a continuous epimorphism $\Phi:H\rightarrow U_c((C_p,[n]))^*$ with discrete kernel. 
\end{enumerate}
\end{lem}
\begin{proof}
 Set $G:=U_c((C_p,[n]))^*$. Let $O_0,\dots,O_{l-1}$ list the orbits of $C_p$ on $[n]$ such that $O_0=\{0\}$ and observe that $l$ is an even number with $l>2$, by our hypotheses. Take $B_{c,\mc{P}}(T_n)$ the partition blow-up of $T_n$ with respect to $\mc{P}:=\{O_0,\dots,O_{l-1}\}$ and the coloring $c$. The group $G$ acts on $B_{c,\mc{P}}(T_n)$, because $C_p$ set-wise fixes the parts of the partition. This action is also without edge inversion, because $G$ acts on $T_n$ without edge inversion.
 
The part $O_0$ consists of only one element, so we modify $B_{c,\mc{P}}(T_n)$ as to ensure that for every vertex $v$, there are $p+2$ edges with origin $v$. To this end, we proceed as follows. Delete each edge $e$ of the form $e=((v,O_0),(w,O_0))$ and add new, distinct edges $e_2,\dots,e_{p+1}$ to $EB_{c,\mc{P}}(T_n)$ such that $o(e_i)=(v,O_0)$ and $t(e_i)=(w,O_0)$. Since $c((v,w))=0$ implies $c((w,v))=0$, we may define $\ol{e_i}:=(\ol{e})_i$. We call the resulting graph $Y$, and this graph is $p+2$ regular. The group $G$ acts on $Y$ via extending the action on $B_{c,\mc{P}}(T_n)$ by declaring that $g(e_i):=(g(e))_i$.  

Fixing $x$ a generator for $C_p$, the element $x$ has a cycle decomposition $s_1\dots s_{l-1}$ where the $s_i$ are pairwise disjoint $p$-cycles and $s_i$ is the $p$-cycle that permutes $O_i$. Fix $h_i\in \mathrm{Sym}([n])$ such that $h_is_ih_i^{-1}$ is the $p$-cycle $(2,\dots,p+1)$. In particular, $h_i(O_i)=\{2,\dots,p+1\}$. We now define a regular coloring $a:EY\rightarrow [p+2]$ of $Y$. Fix a vertex $(v,O_i)\in VY$ and let $f\in E_Y((v,O_i))$. 

\begin{enumerate} 
\item If $i=0$, then
\[
a(f):=
\begin{cases}
k & \text{ if }f=e_k \text{ for some }k\in \{2,\dots, p+1\}\\
1 &  \text{ if } f=((v,O_0),(v,O_{l-1}))\\
0 & \text{ if } f=((v,O_0),(v,O_1))
\end{cases}.
\]
\item If $i= l-1$, then
\[
a(f):=
\begin{cases}
h_{l-1}c((v,w)) & \text{ if }f=((v,O_{l-1}),(w,O_{j}))\text{ with }v\neq w\\
0 & \text{ if } f=((v,O_{l-1}),(v,O_{l-2}))\\
1 &  \text{ if } f=((v,O_{l-1}),(v,O_{0}))
\end{cases}.
\]
\item If $i\neq l-1$ and is odd, then
\[
a(f):=
\begin{cases}
h_ic((v,w)) & \text{ if }f=((v,O_i),(w,O_{j}))\text{ with }v\neq w\\
0 & \text{ if } f=((v,O_i),(v,O_{i-1}))\\
1 &  \text{ if } f=((v,O_i),(v,O_{i+1}))
\end{cases}.
\]
\item If $i\neq 0$ and is even, then
\[
a(f):=
\begin{cases}
h_ic((v,w)) & \text{ if }f=((v,O_i),(w,O_{j}))\text{ with }v\neq w \\
1 & \text{ if } f=((v,O_i),(v,O_{i-1}))\\
0 &  \text{ if } f=((v,O_i),(v,O_{i+1}))
\end{cases}.
\]
\end{enumerate}
The map $a$ is a regular coloring, and furthermore, $a(e)=z$ implies $a(\ol{e})=z$ for $z\in \{0,1\}$. The latter claim follows since $l-1$ is odd. 

Let us compute the local action $\sigma_a(g,(v,O_i))$ for $g\in G$. If $i=0$, then it is immediate that $\sigma_a(g,(v,O_0))=1$. As the remaining cases are similar, we compute the local action $\sigma_a(g,(v,O_i))$ for $i\neq 0$ and even. We see that 
\[
a^{-1}_{(v,O_i)}(k)=
\begin{cases}
((v,O_i),(v,O_{i-1})) & \text{ if }k=1\\
((v,O_i),(v,O_{i+1})) & \text{ if }k=0\\
((v,O_i),(w,O_j)) & \text{ if } c(v,w)=h_i^{-1}(k) \text{ and } k\neq 0,1
\end{cases}.
\]
It follows immediately that $\sigma_a(g,(v,O_i))(k)=k$ for $k\in \{0,1\}$. For $k\notin\{0,1\}$, we see that
\[
\sigma_a(g,(v,O_i))(k)= a\big(((g(v),O_i),(g(w),O_j))\big)
\]
such that $c(v,w)=h_i^{-1}(k)$. The value $c((g(v),g(w))$ equals $\sigma_c(g,v)(c(v,w))$, so 
\[
a((g(v),O_i),(g(w),O_j)))=h_i \sigma_c(g,v)h_i^{-1}(k).
\]
By our choice of $h_i$, we conclude that $h_i \sigma_c(g,v)h_i^{-1}$ acts as some power of the $p$-cycle $(2,\dots,p+1)$ on $[p+2]$. For all $g\in G$ and $(v,O_i)\in VY$, it is thus the case that $\sigma_c(g,(v,O_i))\in (C_p,[p+2])$.

The covering tree of $Y$ is $T_{p+2}$. Let $d$ be the lift of the coloring $a$ to $T_{p+2}$. It follows that $d(e)=i$ if and only if $d(\ol{e})=i$ for $i\in \{0,1\}$ and that $d$ is regular.  Applying Theorem~\ref{thm:graph_cover}, we obtain $H\leq \Aut(T_{p+2})$ the covering group of $G$ and $\Phi:H\rightarrow G$ the covering homomorphism. The subgroup $H$ is closed, and the map $\Phi$ has a discrete kernel. Via Lemma~\ref{lem:colorings_covers}, $\sigma_{d}(g,v)$ is an element of $(C_p,[p+2])$ for all $g\in H$ and $v\in VT_{p+2}$, so in fact $H\leq U_{d}((C_p,[p+2]))$. The group $H$ also cannot invert edges, since $G$ does not, so we conclude that $H\leq U_d((C_p,[p+2]))^*$. The lemma is now verified.
\end{proof}

\begin{lem}\label{lem:3_main_thm}
For $p>2$ a prime, say that $(C_p,[p+2])$ is a permutation group given by the cycle $(2\dots p+1)$ and say that $c$ is a regular coloring of $T_{p+2}$ such that $c(e)=i$ if and only if $c(\ol{e})=i$ for $i\in \{0,1\}$. There is a closed subgroup $H\leq U(C_p)$ admitting a continuous epimorphism $\Phi:H\rightarrow U_c((C_p,[p+2]))^*$ with discrete kernel. 
\end{lem}
\begin{proof}
 Set $G:=U_c((C_p,[p+2]))^*$. Let  $W:=\{0,1\}$ and $U:=\{2,\dots,p+1\}$ and take $B_{c,\mc{P}}(T_{p+2})$ the partition blow-up of $T_{p+2}$ with respect to $\mc{P}:=\{W,U\}$ and the coloring $c$. The group $G$ acts on $B_{c,\mc{P}}(T_{p+2})$, because $C_p$ set-wise fixes the parts of the partition. This action is also without edge inversion, because $G$ acts on $T_{p+2}$ without edge inversion.
 
 The part $W$ consists of only two elements, so we modify $B_{c,\mc{P}}(T_{p+2})$ as to ensure that for every vertex $(v,L)\in VB_{c,\mc{P}}(T_{p+2})$, there are $p+1$ edges with origin $(v,L)$. To this end, we proceed as follows. For each edge $e$ of the form $e=((v,W),(w,W))$ with $c((v,w))=0$, we delete the edge $e$ and add new edges $e_1,\dots,e_{p-1}$ to $EB_{c,\mc{P}}(T_{p+2})$ such that $o(e_i)=(v,W)$ and $t(e_i)=(w,W)$. For the vertex $(v,W)$, there is also an edge $((v,W),(u,W))$ where $c((v,u))=1$. We rename this edge $e_p$. Since $c((v,w))=c((w,v))$ whenever $c((v,w))\in \{0,1\}$, we may define $\ol{e_i}:=(\ol{e})_i$. We call the resulting graph $Z$, and $G$ acts on $Z$ via extending the action on $B_{c,\mc{P}}(T_{p+2})$ by declaring that $g(e_i):=(g(e))_i$.  This action clearly also does not invert edges.
 
Fix $h\in \mathrm{Sym}([p+2])$ such that $h(2\dots p+1)h^{-1}=(1\dots p)$. For $f\in E_Z((v,W))$, we define
\[
a(f):=
\begin{cases}
i & f=e_i\text{ for some }i\in \{1,\dots,p\}\\
0 & f=((v,W),(v,U))
\end{cases}.
\] 
For $f\in E_Z((v,U))$, we define
\[
a(f):=
\begin{cases}
0 & f=((v,U),(v,W))\\
hc((v,w)) & f=((v,U),(w,U))
\end{cases}.
\] 
The map $a$ is a regular coloring of $Z$ by $[p+1]$ and $a(e)=0$ implies $a(\ol{e})=0$. 

Let us now compute the local action $\sigma_a(g,(v,L))$  for $g\in G$. It is immediate that if $L=W$, then $\sigma_a(g,(v,W))=1$. For $L=U$, we note that 
\[
a^{-1}_{(v,U)}(k)=
\begin{cases}
((v,U),(v,W)) & \text{ if }k=0\\
((v,U),(w,U)) & \text{ if }k=hc(v,w) \text{ and } k\neq 0
\end{cases}.
\]
It follows immediately that $\sigma_a(g,(v,U))(0)=0$. For $k\neq 0$,
\[
\sigma_a(g,(v,U))(k)= a\big(((g(v),U),(g(w),U))\big)
\]
such that $c(v,w)=h^{-1}k$. The value $c((g(v),g(w))$ equals $\sigma_c(g,v)(c(v,w))$, so 
\[
a((g(v),U),(g(w),U)))=h\sigma_c(g,v)h^{-1}(k)
\]
By our choice of $h$, we conclude that $h \sigma_c(g,v)h^{-1}$ acts as some power of the $p$-cycle $(1,\dots,p)$ on $[p+1]$. We conclude that for all $g\in G$ and $(v,L)\in VZ$, $\sigma_a(g,(v,L))\in (C_p,[p+1])$ where $(C_p,[p+1])$ is given by the cycle $(1\dots p)$.

The covering tree of $Z$ is $T_{p+1}$. Let $d$ be the lift of the coloring $a$ to $T_{p+1}$. It follows that $d(e)=0$ if and only if $d(\ol{e})=0$.  Applying Theorem~\ref{thm:graph_cover}, we obtain $H\leq \Aut(T_{p+1})$ the covering group of $G$ and $\Phi:H\rightarrow G$ the covering homomorphism. The subgroup $H$ is closed, and the map $\Phi$ has a discrete kernel. Via Lemma~\ref{lem:colorings_covers}, $\sigma_{d}(g,v)$ is an element of $(C_p,[p+1])$ for all $g\in H$ and $v\in VT_{p+1}$, so in fact $H\leq U_{d}((C_p,[p+1]))$. Lemma~\ref{lem:legal_Cp} ensures that $U_{d}((C_p,[p+1]))=U(C_p)$, so the lemma is verified.
\end{proof}

We are now prepared to prove the main technical theorem of this section. The hypotheses of case (2) of the next theorem are satisfied by the Frobenius groups found in Theorem~\ref{thm:forbenius_group}.

\begin{thm}\label{thm:involve_S}
Suppose that $p$ is a prime and $C_p\leq \Sym(p+1)$.
\begin{enumerate}
\item If $p=2$, then  $U(C_p)$ contains a compactly generated closed subgroup $H$ admitting a discrete normal subgroup $D$ such that $H/D$ is isomorphic to $\Aut(T_3)^+$.
\item If $p>2$ and $F\leq \Sym(n)$ is a Frobenius group such that the Frobenius complement is $C_p$ and has index a power of two, then $U(C_p)$ contains a compactly generated closed subgroup $H$ admitting a discrete normal subgroup $D$ such that $H/D$ is isomorphic to $U(F)^+$.
\end{enumerate}
\end{thm}

\begin{proof}
The case of $p=2$ is already established in Proposition~\ref{prop:involve_S_p=2}. Let us then suppose that $p>2$ and $F\leq \Sym(n)$ is a Frobenius group such that the Frobenius complement is $C_p$ and has index a power of two. 

Fix $c_1$ a legal coloring of $T_n$, let $G_1:=U_{c_1}(F)^+$, and set $(C_p,[n]):=(F_{(0)},[n])$. By Lemma~\ref{lem:1_main_thm}, there is a coloring $c_2$ of $T_n$ such that  $U_{c_2}((C_p,[n]))^*$ has a closed compactly generated subgroup $G_2\leq U_{c_2}((C_p,[n]))^*$ admitting a continuous epimorphism $\Phi_1:G_2\rightarrow G_1$ with discrete kernel.  Additionally, $c_2(e)=0$ if and only if $c_2(\ol{e})=0$. 

If $(C_p,[n])$ has two orbits on $[n]$, then $n=p+1$. Lemma~\ref{lem:legal_Cp} supplies a legal coloring $d$ such that $U_{c_2}((C_p,[p+1]))=U_d((C_p,[p+1]))=U(C_p)$. Therefore, $U(C_p)$ contains a compactly generated closed subgroup $G_2$ admitting a discrete normal subgroup $D$ such that $G_2/D$ is isomorphic to $G_1=U(F)$, as required.

Let us henceforth assume that $(C_p,[n])$ has more than two orbits on $[n]$. Since $(F,[n])$ is a Frobenius group, $(C_p,[n])$ has exactly one fixed point, namely $0$. We are thus in a position to apply Lemma~\ref{lem:2_main_thm}. Letting $(C_p,[p+2])$ be the permutation group given by the cycle $(2\dots p+1)$, there is a coloring $c_3$ of $T_{p+2}$ such that  $U_{c_3}((C_p,[p+2]))^*$ has a closed subgroup $H\leq U_{c_3}((C_p,[p+2]))^*$ admitting a continuous epimorphism $\Phi:H\rightarrow U_{c_2}((C_p,[n]))^*$ with discrete kernel. Additionally, $c_3(e)=i$ if and only if $c_3(\ol{e})=i$ for $i\in\{0,1\}$.  We may find a closed compactly generated $G_3\leq H$ such that $\Phi_2:=\Phi\rest_{G_3}:G_3\rightarrow G_2$ is surjective with discrete kernel. 

By Lemma~\ref{lem:3_main_thm}, $U(C_p)$ has a closed subgroup $H$ admitting a continuous epimorphism $\Phi:H\rightarrow U_{c_3}((C_p,[p+2]))^*$ with discrete kernel. We may find a closed compactly generated $G_4\leq H$ such that $\Phi_3:=\Phi\rest_{G_4}:G_4\rightarrow G_3$ is surjective with discrete kernel. The map $\Psi:G_4\rightarrow G_1$ by $\Psi:=\Phi_1\circ \Phi_2\circ \Phi_3:G_4\rightarrow G_1$ is continuous and surjective and has a discrete kernel. We conclude that $U(C_p)$ admits a compactly generated closed subgroup $G_4$ admitting a discrete normal subgroup $D$ such that $G_4/D$ is isomorphic to $U(F)$. The theorem is now verified.
\end{proof}

\begin{cor}\label{cor:involve_S}
Let $F\leq \Sym(d)$ be a permutation group that does not act freely. Then $U(F)$ contains a compactly generated closed subgroup $H$ admitting a discrete normal subgroup $D$ such that $H/D$ is topologically simple and non-discrete. In particular, $U(F)$ admits a subquotient in $\ms{S}$.
\end{cor}
\begin{proof}
Immediate from Lemma~\ref{lem:PrimeReduction} and Theorem~\ref{thm:involve_S}. 
\end{proof}

\subsection{Elementary groups and relative Tits cores}\label{sec:TitsCore}

The class of elementary groups, denoted by $\ms{E}$, is the smallest class of totally disconnected locally compact second countable (\tdlcsc) groups that contains the second countable profinite groups and the countable discrete groups and that is closed under the elementary operations; see \cite{W14}. (These operations are taking closed subgroups, Hausdorff quotients, group extensions, and countable directed unions of open subgroups.) The class of elementary groups is disjoint from the class $\ms{S}$ comprising the non-discrete compactly generated topologically simple \tdlcsc groups.

We say that a topological group $H$ admits a group $G$ as a \textbf{subquotient} if there is some closed subgroup $K\leq H$ such that $K$ admits $G$ as a continuous quotient. Admitting a subquotient a group in $\ms{S}$ is sufficient to be non-elementary; see \cite{W14}. The following consequence of Corollary~\ref{cor:involve_S}  implies that  the non-discrete Burger--Mozes universal groups $U(F)$ all admit groups in $\ms{S}$ as subquotients and are thus non-elementary.

\begin{cor}\label{cor:FreeF} 
	For $F$ a finite permutation group, the following are equivalent:
	\begin{enumerate}
		\item $F$ acts freely.
		\item $U(F)$ is discrete.
		\item $U(F)$ is elementary.	
		\item $U(F)$ does not admit a group in $\ms{S}$ as a subquotient.
	\end{enumerate}
\end{cor}
\begin{proof}
	That $(1)$ implies $(2)$ follows from Proposition~\ref{prop:BuMoSimple}(i), and $(2)$ implies $(3)$ is immediate. The contrapositive of $(3)$ implies $(4)$ is given by \cite[Proposition 6.5]{W14}. Finally, Corollary~\ref{cor:involve_S} gives the contrapositive of $(4)$ implies $(1)$
\end{proof}

All known examples of non-elementary groups are thus because they admit some group in $ \ms{S}$ as a subquotient. One naturally asks if admitting a subquotient in $\ms{S}$ is  a necessary condition to be non-elementary.

\begin{qu}\label{qu:involve_S} 
	For $G$ a \tdlcsc group, if $G$ is non-elementary, then is there a compactly generated closed $H\leq G$ such that $H$ has a continuous quotient in $\ms{S}$?
\end{qu}

One possible approach to Question~\ref{qu:involve_S} is via a stronger formulation due to Reid. The positive answer to Reid's question, Question~\ref{qu:reid} below, implies the positive answer to Question~\ref{qu:involve_S}. For a locally compact group $G$, recall that $g\in G$ is \textbf{periodic} if $\grp{g}$ is relatively compact. For any element $g \in G$, we define the \textbf{contraction group} of $g $   as 
$$\con(g) := \{x \in G \mid \lim_n g^n x g^{-n} = 1\}.$$
The \textbf{relative Tits core} of $g$ in $G$, denoted by $G^\dagger_g$, is defined by 
$$G^\dagger_g := \overline{\langle \con(g) \cup \con(g^{-1})\rangle}.$$

\begin{qu}[Reid, {\cite[Question 2]{R16}}]\label{qu:reid}
	Let $G$ be a \tdlcsc group. If $G$ is non-elementary, then is there $g\in G$ non-periodic and $n\geq 1$ such that $g^n$ is an element of the closure of the relative Tits core of $g$? 
\end{qu}

The Burger--Mozes groups $U(F)$ with $F$ nilpotent provide examples  demonstrating that the answer to Question~\ref{qu:reid} is negative. 

\begin{cor} 
For  $F$ nilpotent, every non-periodic element $g\in U(F)$ is such that $g^n\notin G^{\dagger}_g$ for all $n\neq 0$. 
\end{cor}
\begin{proof}
In any locally compact group $G$, the group $\cgrp{G^{\dagger}_g,g}$ is compactly generated for any $g \in G$.  Let now $G = U(F)$ and let $g\in U(F)$ be a non-periodic element. The group $H:=\cgrp{G^{\dagger}_g,g}$ is thus a compactly generated closed subgroup of $U(F)$, and $g$ is hyperbolic. Appealing to Theorem~\ref{thm:Indic}, $H$ admits an infinite discrete quotient $H/O$. Since $G^{\dagger}_g$ is topologically generated by contraction groups, it follows that $G^{\dagger}_g\leq O$. We deduce that $H/\ol{G^{\dagger}_g}$ is infinite, and thus, $g^n\notin \ol{G^{\dagger}_g}$ for any non-zero power of $g$.
\end{proof}

If in addition $F$ does not act freely, then $U(F)$ is non-elementary by Corollary~\ref{cor:FreeF}, so  we obtain the following consequence of Theorem~\ref{thm:Indic} yielding a negative answer to Question~\ref{qu:reid}.
 
\begin{cor} For any nilpotent permutation group $(F,\Omega)$ such that $F$ does not act freely on $\Omega$, the Burger--Mozes universal group $U(F)$ is non-elementary, and every non-periodic element $g\in U(F)$ is such that $g^n\notin G^{\dagger}_g$ for all non-zero $n$.
\end{cor}

\section{Lattices}

\subsection{Intersecting lattices with subgroups}

The following basic facts are well-known.

\begin{prop}\label{prop:Raghu}
Let $G$ be a locally compact group, $O \leq G$ be an open subgroup and $H \leq G$ be a closed subgroup. 
\begin{enumerate}[(i)]
\item If $H$ is cocompact in $G$, then $H \cap O$ is cocompact in $O$. 

\item If $H$ is of finite covolume in $G$, then $H \cap O$ is of finite covolume in $O$. 

\item If $H$ is a lattice in $G$, then $H \cap O$ is a lattice in $O$. 
\end{enumerate}	
\end{prop}
\begin{proof}
All assertions follow by observing that the openness of $O$ implies that the image of the canonical projection 	
$$O / H\cap O \to G/H$$
is both open and closed. 
\end{proof}
	
\begin{lem}[{\cite[Lemma~I.1.7]{Raghu}}]\label{lem:LatticeInNormal}
Let $G$ be a locally compact group and $H, N \leq G$ be closed subgroups. Assume that $N$ is normal and that $HN$ is closed. 
\begin{enumerate}[(i)]
	\item  $HN/N$ is cocompact in $G/N$ if and only if $H \cap N$ is cocompact in $N$.
	
	\item $HN/N$ is of finite covolume in $G/N$ if and only if $H \cap N$ is of finite covolume in $N$.
\end{enumerate}	
\end{lem}

\subsection{The amenable radical and centralizers of  lattices}

The following fundamental fact is essentially due to H. Furstenberg; see for example \cite[Lemma 16.7]{Fur73}. For completeness, we give a proof via the main result of \cite{BaDuLe}. 
\begin{thm}\label{thm:RadAmen}
	Let $G$ be a second countable locally compact group. If $G$ has a trivial amenable radical, then every closed subgroup of finite covolume has a trivial amenable radical. 
\end{thm}	
\begin{proof}
Let $H$ be a closed subgroup of finite covolume in $G$ and $R \leq H$ be the amenable radical of $H$. The orbit map under the conjugation action of $G$ induces a continuous surjective map $G/H \to \{gRg^{-1} \mid g \in G\} \subset \Sub(G)$. Pushing forward the finite invariant measure on $G/H$ we get an amenable IRS of $G$. The main result of \cite{BaDuLe} ensures that this IRS is contained in the amenable radical of $G$, which is trivial by hypothesis. We conclude that $R$ is trivial. 
\end{proof}

The following algebraic consequences will be useful. 

\begin{cor}\label{cor:CentLattice}
Let $G$ be a $\sigma$-compact locally compact group with a trivial amenable radical. Then every closed subgroup $H \leq G$ of finite covolume has a trivial centralizer.  
\end{cor}

\begin{proof}
By \cite[Theorem 8.7]{HR}, the group $G$ has a   compact normal subgroup $K$ such that $G/K$ is second countable. Since $G$ has a trivial amenable radical, the group $K$ is trivial, hence $G$ is second countable. 
Let $a \in \cent_G(H)$ and set $A = \overline{\langle a \rangle}$ and $B = \overline{AH}$. The subgroup $B$ is closed, is of finite covolume in $G$, and contains $A$ in its center. Theorem~\ref{thm:RadAmen} implies that $A$ is trivial. We conclude that $\cent_G(H)$ is trivial. 
\end{proof}

\begin{cor}\label{cor:NormLattice}
	Let $G$ be a  locally compact group with a trivial amenable radical. If $\Gamma \leq G$ is a finitely generated lattice, then $\norm_G(\Gamma)/\Gamma$ is finite. In particular, $\norm_G(\Gamma)$ is a finitely generated lattice.
\end{cor}
\begin{proof}
Since $G$ contains a finitely generated lattice, it is compactly generated by \cite[Lemma~2.12]{CaMo_discrete}. In particular, $G$ is $\sigma$-compact. By Corollary~\ref{cor:CentLattice}, we have $\cent_G(\Gamma)=\{1\}$, so the natural continuous map $\norm_G(\Gamma) \to \Aut(\Gamma)$ is injective. As $\Gamma$ is finitely generated, its automorphism group is  countable, so $\norm_G(\Gamma)$ is a countable locally compact group. Hence, $\norm_G(\Gamma)$ is discrete by the Baire category theorem. Since  $\Gamma$ is of finite covolume in $G$, it is of finite covolume in $\norm_G(\Gamma)$ (see \cite[Lemma~I.1.6]{Raghu}), hence  it is of finite index in  $\norm_G(\Gamma)$. The required assertions now follow. 
\end{proof}

\subsection{The discrete residual and the quasi-center}

The \textbf{discrete residual} of a topological group $G$, denoted by $\Res(G)$, is the intersection of all open normal subgroups of $G$. A group whose discrete residual is trivial is called \textbf{residually discrete}. Residually discrete groups are exactly the groups such that each non-trivial element is non-trivial in some discrete quotient. 

The \textbf{quasi-center} of a locally compact group $G$, denoted by $\QZ(G)$, is the set of elements whose centralizer is open. It is a characteristic (but not necessarily closed) subgroup of $G$ containing all discrete normal subgroups.

\begin{lem}\label{lem:ResProduct}
For $G_1, G_2$ locally compact groups, we have 
$$\Res(G_1 \times G_2) = \Res(G_1) \times \Res(G_2)$$ 
and
$$ \QZ(G_1 \times G_2) = \QZ(G_1) \times \QZ(G_2).$$ 
\end{lem}
\begin{proof}
	If $N_1$ and $N_2$ are open normal subgroups of $G_1$ and $G_2$ respectively, then $N_1 \times N_2$ is open and normal in $G_1 \times G_2$. This implies that $\Res(G_1 \times G_2) \leq \Res(G_1) \times \Res(G_2)$. Conversely, let $N \leq G_1 \times G_2$ be an open normal subgroup. There is a compact identity neighborhood $K_i$ in $G_i$ such that $K_1 \times K_2$ is contained in $N$. Let $N_i$ be the smallest normal subgroup of $G_i$ containing $K_i$. Since the subgroup $N_i$ contains an identity neighborhood of $G_i$, it is open.  Since $K_i \times \triv \subset N$, we have $N_i \times \triv \subset N$, and thus $N_1 \times N_2 \leq N$. We conclude that $\Res(G_1 \times G_2) \geq \Res(G_1) \times \Res(G_2)$.
	
	For the quasi-center, notice that the inclusion $\QZ(G_1 \times G_2) \geq \QZ(G_1) \times \QZ(G_2)$ is clear. The reverse inclusion follows from the fact that the canonical projections of $G_ 1 \times G_2$ to $G_1$ and $G_2$ are open maps. 
\end{proof}

We shall need two further facts.

\begin{prop}[{\cite[Corollary 4.1]{CM11}}]\label{prop:ResDis}
	A compactly generated totally disconnected locally compact group is residually discrete if and only if it has a basis of identity neighborhoods consisting of compact open normal subgroups. 
\end{prop}
	
\begin{cor}\label{cor:CentraDiscRes}
Let $G$ be a compactly generated locally compact with a trivial amenable radical. If every open normal subgroup of $G$ is of finite index, then the discrete residual $\Res(G)$ has a trivial centralizer. 
\end{cor}
\begin{proof}
By Proposition~\ref{prop:ResDis}, the compactly generated residually discrete (hence totally disconnected) quotient group $G/\Res(G)$ is compact-by-discrete. Since $G$ does not have any infinite discrete quotient, the same holds for 	$G/\Res(G)$, and we infer that $G/\Res(G)$ is compact. In particular $G/\Res(G)$ carries a $G$-invariant measure of finite volume. In other words $\Res(G)$ is a closed subgroup of finite covolume in $G$. Its centralizer is thus trivial by Corollary~\ref{cor:CentLattice}.
\end{proof}

\subsection{Residually finite irreducible lattices in products}

A fundamental discovery of Burger and Mozes \cite{BuMo2} is that an irreducible lattice in the product of two trees which is residually finite must have injective projection to both factors under natural, mild conditions. That fact is generalized to lattices in products of groups acting on CAT($0$) spaces in \cite[\S2.B]{CaMo_KM}.  In this section, we present a purely algebraic version of that fact, which is based on similar arguments.  

Throughout this section, we let $G_1, \dots, G_n$ be non-trivial locally compact groups and denote by 
$$\pi_i \colon G_1 \times \dots \times G_n \to G_i$$ 
the canonical projection to the $i$-th factor. 

A   group is called \textbf{Noetherian} if it satisfies the ascending chain condition on   subgroups. A group is Noetherian if and only if all its  subgroups are finitely generated.  Finite groups are obvious examples of Noetherian groups.

\begin{lem}\label{lem:ResNoeth}
	Let $\Gamma \leq G_1 \times G_2$ be a lattice such that $\pi_1(\Gamma)$ is dense in $G_1$.   	Assume that the discrete residual $\Res(G_1)$ has a trivial centralizer in $G_1$. If $\Gamma$ is residually Noetherian (e.g. residually finite), then the restriction $\pi_2 \rest_\Gamma \colon \Gamma \to G_2$ is injective. 
\end{lem}
\begin{proof}
The kernel of $\pi_2 \rest_\Gamma$ is a discrete subgroup of $G_1 \times G_2$ of the form $N_1 \times \triv$ for some subgroup $N_1 \leq G_1$. In particular, $N_1$ is discrete in $G_1$. Since $N_1 \times \triv$  is normal in $\Gamma$, we see that $N_1$ is normalized by $\pi_1(\Gamma)$. The normalizer of $N_1$ is closed, so $N_1$ is in fact normal in $G_1 = \overline{\pi_1(\Gamma)}$. 

Let now $M$ be a normal subgroup of $\Gamma$ such that $\Gamma/M$ is Noetherian. Observe that $M \cap (N_1 \times \triv)$ is a discrete subgroup of $G_1 \times G_2$  of the form $M_1 \times \triv$ with $M_1 \leq G_1$ discrete. Moreover, since $M$ and $N_1 \times \triv$ are normal in $\Gamma$, it follows that $M_1$ is normal in $G_1 = \overline{\pi_1(\Gamma)}$. The quotient 
\[
N_1/M_1 \cong N_1 \times \triv/M \cap (N_1 \times \triv)
\]
is isomorphic to a subgroup of $\Gamma/M$ and is thus finitely generated since $\Gamma/M$ is Noetherian.  Hence, its automorphism group is finite or countable. 

The conjugation action of $G_1$ on $N_1/M_1$ yields a continuous homomorphism $G_1 \to \Aut(N_1 /M_1)$. The kernel of that homomorphism is a closed subgroup of $G_1$ of finite or countable index, so the kernel is open by the Baire category theorem. This implies  that $[\Res(G_1), N_1] \leq M_1$. We thus have
\[
[\Res(G_1), N_1] \times \triv \leq M_1 \times \triv \leq M.
\]
Since this is valid for   all normal subgroups $M$ of $\Gamma$  such that $\Gamma/M$ is Noetherian,  that $\Gamma$ is residually Noetherian implies that $[\Res(G_1), N_1] = \triv$. We conclude that $N_1$ is contained in the centralizer of $\Res(G_1)$, which is trivial by hypothesis. 
\end{proof}

\subsection{Normalizers of compact open subgroups}

\begin{lem}\label{lem:DiscreteProj}
	Let $\Gamma \leq G_1 \times G_2$ be a lattice. Assume that $G_2$ is totally disconnected and that at least one of the following conditions hold:
\begin{enumerate}[(1)]
	\item $\Gamma$ is cocompact, $G_2$ is compactly generated, and every cocompact lattice in $G_2$ has a trivial centralizer in $G_2$. 
	
	\item $G_2$ has Kazhdan's property (T), and every lattice in $G_2$ has a trivial centralizer. 
\end{enumerate}
If   $\overline{\pi_1(\Gamma)}$ has a compact open normal subgroup (e.g. $\pi_1(\Gamma)$ is discrete in $G_1$), then $\pi_2(\Gamma)$ is discrete in $G_2$. 
\end{lem}
\begin{proof}
The group $O_1 := 	\overline{\pi_1(\Gamma)}$ has a compact open normal subgroup $K_1$.  Additionally, $\Gamma$   is a lattice in  $O_1 \times G_2$ which is cocompact if $\Gamma$ is cocompact in $G_1\times G_2$. The natural projection $\varphi\colon O_1 \times G_2 \to O_1/K_1 \times G_2$ is proper since $K_1$ is compact. Thus, $\Lambda := \varphi(\Gamma)$ is a lattice in $O_1/K_1 \times G_2$. Since $O_1/K_1$ is discrete, it follows from Lemma~\ref{lem:LatticeInNormal} that $\Lambda_2 := \Lambda \cap (\triv \times G_2)$ is a lattice in $\triv \times G_2$, and $\Lambda_2$ is cocompact in $\{1\}\times G_2$ if $\Gamma$ is cocompact. 

Let $U_2$ be a compact open subgroup of $G_2$ such that $\triv \times U_2$ intersects $\Lambda $ trivially. The subgroup $\Lambda_1 := \Lambda \cap (O_1/K_1 \times U_2)$ is a lattice in  $O_1/K_1 \times U_2$ by Proposition~\ref{prop:Raghu}. Since $\Lambda_2$ is a normal subgroup of $\Lambda$, we see that $\Lambda_1  \Lambda_2$ is a discrete subgroup. Moreover,  $\Lambda_1\Lambda_2$ is a lattice in $\Lambda_1(\triv \times G_2)$, and $\Lambda_1(\triv \times G_2)$ is of finite covolume in $(O_1/K_1) \times G_2$. We conclude that  $\Lambda_1\Lambda_2$ is a lattice in $(O_1/K_1) \times G_2$, in view of \cite[Lemma~I.1.6]{Raghu}, hence $\Lambda_1\Lambda_2$ is  of finite index in $\Lambda$. 

Let $V_2 \leq U_2$ be the closure of the projection of $\Lambda_1$ to $G_2$ and $\Lambda'_2$ be the projection of $\Lambda_2$ to $G_2$. The group $V_2$ is compact and normalizes $\Lambda'_2$, which is discrete. If hypothesis (1) holds, then $\Lambda'_2$ is a cocompact lattice in the compactly generated group $G_2$ and is thus finitely generated with a trivial centralizer in $G_2$. If hypothesis (2) holds, then $\Lambda'_2$ is also finitely generated with a trivial centralizer in $G_2$.

The compact group $V_2$ acts by conjugation on the finitely generated group $\Lambda'_2$, so $V_2$ admits an open subgroup of finite index which centralizes $\Lambda'_2$. This subgroup must be trivial since $\cent_{G_2}(\Lambda'_2)= \triv$, so $V_2$ is finite. We conclude that $\Lambda_1$ has finite image in $G_2$. Therefore, $\Lambda_1\Lambda_2$ has discrete image in $G_2$, so $\Lambda$ has discrete image in $G_2$ as $[\Lambda: \Lambda_1\Lambda_2]$ is finite. The projection $\pi_2 \colon O_1 \times G_2 \to G_2$ factorizes through $\varphi \colon  O_1 \times G_2 \to O_1/K_1 \times G_2$. Hence,  $\Gamma$ and $\Lambda$ have the same image in $G_2$. That is to say, $\pi_1(\Gamma)$ is discrete.
\end{proof}
	
\begin{lem}\label{lem:CompactNormalizer}
	Let $\Gamma \leq G_1 \times G_2$ be a lattice. Assume that $G_2$ is totally disconnected and that at least one of the following conditions hold:
	\begin{enumerate}[(1)]
	\item $\Gamma$ is cocompact, $G_2$ is compactly generated, and every cocompact lattice in $G_2$ has a trivial centralizer in $G_2$. 
	
	\item $G_2$ has Kazhdan's property (T), and every lattice in $G_2$ has a trivial centralizer. 
	\end{enumerate}
If $\pi_2 \rest_\Gamma \colon \Gamma \to G_2$ is injective, then every compact open subgroup of $G_1$ has a compact normalizer in $G_1$. 
\end{lem}
\begin{proof}
Let $K_1 \leq G_1$ be a compact open subgroup and set $O_1 := \norm_{G_1}(K_1)$. The intersection $\Gamma_{O_1}:= \Gamma \cap (O_1 \times G_2)$ is a lattice in  $O_1 \times G_2$	by Proposition~\ref{prop:Raghu}, and it is additionally cocompact if $\Gamma$ is cocompact. Notice that $K_1 \cap \overline{\pi_1(\Gamma_{O_1})}$ is a compact open normal subgroup of $\overline{\pi_1(\Gamma_{O_1})}$. In view of Lemma~\ref{lem:DiscreteProj}, the projection of $\Gamma_{O_1}$ to $G_2$ is discrete. Applying Lemma~\ref{lem:LatticeInNormal}, $\Gamma_{O_1} \cap (O_1 \times \triv)$ is a lattice in $O_1 \times \triv$. On the other hand, $\pi_2\rest_{\Gamma_{O_1}}$ is injective, hence the intersection $\Gamma_{O_1} \cap (O_1 \times \triv)$ is trivial. The trivial group is then a lattice in $O_1 \times \triv$, which implies that $O_1$ is compact.
\end{proof}

 Notice that   if $\Gamma \leq G_1 \times   G_2$ is a lattice such that $\overline{\pi_1(\Gamma)}= G_1$, then the kernel $\Ker(\pi_2\rest_\Gamma )$ is contained in $\QZ(G_1)$. In particular, if $G_1$ has a trivial quasi-center, then the projection of $\Gamma$ to $G_2$ is injective. The following partial converse   was observed in a conversation with Marc Burger. 

\begin{lem}\label{lem:QZ}
	Let $\Gamma \leq G_1 \times   G_2$ be a lattice whose projection to $G_1$ has dense image.  Assume that $G_1$ has a trivial amenable radical, that $G_1 \times G_2$ is totally disconnected, and that at least one of the following conditions hold:
	\begin{enumerate}[(1)]
		\item $\Gamma$ is cocompact, $G_2$ is compactly generated, and every cocompact lattice in $G_2$ has a trivial centralizer in $G_2$. 
		
		\item $G_2$ has Kazhdan's property (T), and every lattice in $G_2$ has a trivial centralizer. 
	\end{enumerate}

Then $\QZ(G_1)=1$ if and only if the projection $\pi_2 \rest_\Gamma \colon \Gamma \to G_2$ is injective. 
\end{lem}

\begin{proof}
	The `only if' part is straightforward and was observed above. Suppose conversely that $\QZ(G_1)$ is nontrivial. There exists a finite subset $\Sigma \subset \QZ(G_1)$ such that $\overline {\langle \Sigma \rangle}$ is not compact, since otherwise $\QZ(G_1)$, hence also $\overline{\QZ(G_1)}$, is locally elliptic which contradicts the hypothesis that $G_1$ has a trivial amenable radical. By the definition of the quasi-center, the centralizer $C_{G_1}(\Sigma)$ is open, so it contains a compact open subgroup $U$. The normalizer $N_{G_1}(U)$ thus contains the non-compact closed group $\overline {\langle \Sigma \rangle}$. Applying Lemma~\ref{lem:CompactNormalizer}, the projection $\pi_2 \rest_\Gamma \colon \Gamma \to G_2$ is not injective. 
\end{proof}	
	
Combining some of the previous results, we obtain the following criterion.

\begin{prop}\label{prop:RF->CompactNorm}
	Let $\Gamma \leq G_1 \times   G_2$ be a lattice.  
	Assume that the following conditions hold:
	\begin{itemize}
		\item $\overline{\pi_1(\Gamma)}= G_1$.
				
		\item  $C_{G_1}(\Res(G_1)) =\triv$.

		\item $G_2$ is compactly generated,  totally disconnected with a trivial amenable radical.

		\item  $\Gamma$ is   cocompact in $G_1\times G_2$, or  $G_2$ has Kazhdan's property (T). 
	\end{itemize} 
	If a compact open subgroup of $G_1$ has a non-compact normalizer in $G_1$, then the projection of $\Gamma$ to $G_2$ is not injective, and $\Gamma$ is not residually Noetherian (in particular not residually finite).
\end{prop}

\begin{proof}
	By Corollary~\ref{cor:CentLattice}, every lattice in $G_2$ has a trivial centralizer. If a compact open subgroup of $G_1$ has a non-compact normalizer, then  the restriction $\pi_2 \rest_\Gamma \colon \Gamma \to G_2$ is  not injective by  Lemma~\ref{lem:CompactNormalizer}. Moreover, the   hypotheses of Lemma~\ref{lem:ResNoeth} are  all fulfilled, so $\Gamma$ is not residually Noetherian.  
\end{proof}

\begin{thm}\label{thm:MultipleProduct}
	Let $G = G_1 \times \dots \times G_n$ be a product of non-discrete compactly generated totally disconnected locally compact groups with a trivial amenable radical. Suppose that for each $i$, every open normal subgroup of $G_i$ is of finite index.  Let $\Gamma \leq G$ be a lattice.  such that  $\overline{\pi_i(\Gamma)}= G_i$ for all $i$. If $\Gamma $ is not cocompact, we assume in addition that $G$ has Kazhdan's property (T). If $\Gamma$ is residually Noetherian, then $\QZ(G) =\triv$, and for all $i$, every compact open subgroup of $G_i$ has a compact normalizer in $G_i$. 	
\end{thm}
\begin{proof}
	Fix $i \in \{1, \dots, n\}$. Set $H_1: = G_i$ and $H_2: = \prod_{j \neq i} G_j$. Every open normal subgroup of $H_1$ is of finite index, so $C_{H_1}(\Res(H_1))=\triv$ by   Corollary~\ref{cor:CentraDiscRes}. 
	Therefore $\Gamma \leq H_1 \times H_2$ is a lattice fulfilling all the hypotheses of Proposition~\ref{prop:RF->CompactNorm}. It follows that every compact open subgroup of $H_1 =G_i$ has a compact normalizer in $G_i$. 	 Moreover, the projection of $\Gamma$ to $H_2$ is injective by Lemma~\ref{lem:ResNoeth}, so $\QZ(H_1) = \QZ(G_i) = \triv$ by Lemma~\ref{lem:QZ}. Appealing to Lemma~\ref{lem:ResProduct}, we deduce that $\QZ(G)=\triv$.  
\end{proof}

The following related result describes a similar property under a stronger hypothesis of irreducibility on the lattice $\Gamma$.

\begin{thm}\label{thm:ResNoethMultiple}
	Let $G = G_1 \times \dots \times G_n$ be a product of non-discrete locally compact groups and $\Gamma\leq G$ be a lattice such that for each $i$, the projection of $\Gamma$ to $\prod_{j \neq i} G_j$ has dense image. 

\begin{enumerate}[(i)]
	\item Assume that $\cent_{G_i}(\Res(G_i)) = \triv$ for all $i$. If $\Gamma$ is residually Noetherian, then  the restriction $\pi_i \rest_\Gamma \colon \Gamma \to G_i$ is injective for all $i$. 
	
	\item Assume that $G$ is compactly generated, totally disconnected with a trivial amenable radical. If $\Gamma$ is not cocompact, assume in addition that $G$ has Kazhdan's property (T). Then $\QZ(G) = \triv$ if and only if the projection of $\Gamma$ to $G_i$ is injective for all $i$. 
	
\end{enumerate}	
\end{thm}
\begin{proof}
Let $i \in \{1, \dots, n\}$ and set $H: = \prod_{j \neq i} G_j$. 

\medskip \noindent (i). 
By hypothesis, $\Gamma$ is a lattice with dense projections in the product $H \times G_i$. In view of Lemma~\ref{lem:ResProduct}, we additionally have $\cent_H(\Res(H)) = \triv$. Therefore, $\pi_i \rest_\Gamma \colon \Gamma \to G_i$ is injective by Lemma~\ref{lem:ResNoeth}.

\medskip \noindent (ii).
By assumption, $H$ and $G_i$ both have a trivial amenable radical. By   Corollary~\ref{cor:CentLattice}, every lattice in $G_i$ has a trivial centralizer. It follows from Lemma~\ref{lem:QZ} that $\QZ(H) = \triv$ if and only if the projection of $\Gamma$ to $G_i$ is injective. The required assertion now follows from Lemma~\ref{lem:ResProduct}.
\end{proof}

\section{Trees and lattices in products}

The goal of this section is to apply our abstract results on lattices in product groups to the geometric setting of groups acting on trees.  We first identify a local criterion controlling the normalizers of compact open subgroups. 

\subsection{Normalizers of compact open subgroups}

\begin{prop}\label{prop:CompactNorm:1}
	Let $T$ be a locally finite tree with more than two ends and $G \leq \Aut(T)$ be a closed unimodular subgroup acting cocompactly. For each vertex $v \in VT$ and each edge $e \in E(v)$, suppose that the stabilizer $G_{(e)}$ fixes an edge $f \in E(v)$ different from $e$. Then there is an edge $e \in ET$ whose stabilizer $G_{(e)}$ has a non-compact normalizer.
\end{prop}

\begin{proof}		
	By induction, we build a geodesic edge path $(f_n)_{n\geq 0}$ such that $G_{(f_n)}\leq G_{(f_{n+1})}$. As the base and successor cases are the same, suppose we have built our sequence up to $f_k$. Noting that $G_{(\ol{f}_k)}=G_{(f_k)}$, our hypothesis ensures that there exists $f_{k+1}\in E(t(f_k))$ different from $\bar f_k$ such that $G_{(f_k)}$ fixes $f_{k+1}$.  That is to say, $G_{(f_k)}\leq G_{(f_{k+1})}$. This completes our inductive construction.
	
	Since $G$  acts cocompactly on $T$, the collection of edges $\{f_n\}_{n \geq 0}$ is covered by finitely many edge orbits. In particular, there is an infinite subset $I \subset \Nb$ such that $\{f_n\}_{n \in I}$ is contained in the same orbit. For $m\leq n$ with $m,n\in I$, the compact open subgroups $G_{(e_n)}$ and $ G_{(e_m)}$  are conjugate. Since $G$ is unimodular, we conclude that $G_{(f_n)}=G_{(f_m)}$. For $n \in I$, the normalizer of $G_{(f_n)}$ therefore contains elements mapping $o(f_n)$ arbitrarily far away from itself. The normalizer is thus non-compact.
\end{proof}

The hypothesis of Proposition~\ref{prop:CompactNorm:1} is satisfied in the following  situation. 

\begin{cor}\label{cor:non-compact-norm}
	Let $T$ be a locally finite   tree with more than two ends and $G \leq \Aut(T)$ be a closed unimodular subgroup acting cocompactly on $T$. Suppose that for each vertex $v \in VT$,  the local action of $G_{(v)}$ on $E(v)$ is nilpotent and does not have a unique fixed point. 
	 	Then $G$ has a compact open subgroup with a non-compact normalizer.
\end{cor}
\begin{proof}
	Let $F$ be the natural image of $G_{(v)}$  in  $\Sym(E(v))$, so $F$ is nilpotent by hypothesis. In view of Lemma~\ref{lem:CompactNormalizer}, it suffices to show that for each $e \in E(v)$, there exists $f \in E(v)$ different from $e$ fixed by $F_{(e)}$. If $e$ is not a fixed point of $F$ in $E(v)$, then  $F_{(e)}$ is a proper subgroup of $F$. Since $F$ is nilpotent, it follows that $F_{(e)}$ is properly contained in its normalizer. Taking $g \in F$ normalizing $F_{(e)}$ without fixing $e$, $f:=g(e)$ is a fixed point of $F_{(e)}$ different from $e$. If $e$ is a fixed point of $F$, then by hypothesis there exists $f \in E(v)$ different from $e$ which is also fixed by $F$. The conclusion now follows from  Proposition~\ref{prop:CompactNorm:1}.
\end{proof}

We pause to note a supplementary result in the same vein, which applies in particular to all Burger--Mozes groups $U_c(F)$ with $F$ nilpotent.

\begin{prop}\label{prop:CompactNorm:2}
	Let $T$ be a locally finite   tree with more than two ends and $G \leq \Aut(T)$ be a closed unimodular subgroup acting vertex-transitively on $T$. Suppose that for each vertex $v \in VT$,  the local action of $G_{(v)}$ on $E(v)$ is nilpotent. 	Then $G$ has a compact open subgroup with a non-compact normalizer.
	
\end{prop}
\begin{proof}
	Let $v \in VT$ and $F \leq \Sym(E(v))$ denote the local action of $G_{(v)}$ at 	$v$. If $F$ does not have a unique fixed point, then we may apply Corollary~\ref{cor:non-compact-norm}, since $G$ is vertex-transitive, and the required conclusion follows. We assume henceforth that $F$ has a unique fixed point, say $e$.

	By Proposition~\ref{prop:BuMo}(ii), there exists a regular, but not necessarily legal, coloring $c$ of $T$ such that  $G$ is contained as a closed subgroup in $U_c(F)$.  Moreover, since $G$ is unimodular and vertex-transitive, we deduce from Lemma~\ref{lem:LocalAction-UniqueFixedPoint} that $c(e) = c(\bar e)$, where $e$ is the unique edge fixed by the local action.
	
	Let now $w \in VT$ be an arbitrary vertex and $e_w \in E(w)$ be the unique edge with $c(e_w)= c(e)$. For any edge $f \in E(w)$ different from $e_w$, the edge-stabilizer $G_{(f)}$ fixes an edge $f' \in E(w)$ which is different from both $f$ and $e_w$. Proceeding as in the proof of Proposition~\ref{prop:CompactNorm:1}, we can construct inductively a geodesic edge path $(f_n)_{n\geq 0}$ such that $G_{(f_n)}\leq G_{(f_{n+1})}$ with $c(f_n) \neq c(e) \neq c(\bar{f_n})$ for all $n$. The end of the proof is identical to that of Proposition~\ref{prop:CompactNorm:1}. 
\end{proof}

\subsection{Application to lattices in products of trees}

We now obtain the following  criterion on the local action ensuring that some  lattices are not residually finite.  

\begin{cor}\label{cor:NonRF}
Let  $T$ be a locally finite leafless tree  such that $\Aut(T)$ acts cocompactly, let   $H$ be a compactly generated totally disconnected locally compact group with a trivial amenable radical, and let  $\Gamma \leq \Aut(T) \times H$ be a cocompact lattice. Assume that at least one of the following conditions is satisfied:
\begin{enumerate}[(1)]
	\item For all $v \in VT$ and $e \in E(v)$, the stabilizer $\Gamma_{(e)}$ fixes an edge $f\in E(v)$ different from $e$. (E.g. the natural image of $\Gamma_{(v)}$ in $\Sym(E(v))$ is nilpotent and without a unique fixed point.) 
	
	\item The $\Gamma$-action on $T$ is vertex-transitive, and for every $v \in VT$, the local action of $\Gamma$ at $v$ is nilpotent. 
\end{enumerate}
If the projection of $\Gamma$ to $H$ is non-discrete, then $\Gamma$ is not residually Noetherian, hence not residually finite. 
\end{cor}
\begin{proof}
	Let $G_1 \leq \Aut(T) $ denote the closure of the projection of $\Gamma$. The group $G_1$ acts cocompactly on $T$, since $\Gamma$ is cocompact, hence $G_1$ acts minimally. Additionally, $G_1$ is unimodular because $G_1 \times H$ contains a lattice. In particular, $G_1$ does not fix an end of $T$. It follows from Proposition~\ref{prop:Tits} that every non-trivial normal subgroup of $G_1$ has a trivial centralizer and is non-amenable. By Corollary~\ref{cor:CentLattice}, every lattice in $H$ has a trivial centralizer. Since the projection of $\Gamma$ to $H$  is non-discrete by hypotheses, it follows from  Lemma~\ref{lem:DiscreteProj} that $G_1$ is non-discrete.  
	
By Proposition~\ref{prop:CompactNorm:1} or Proposition~\ref{prop:CompactNorm:2}, either of the hypotheses (1) or (2) implies the existence of  a compact open subgroup $K \leq  G_1$ whose normalizer in $G_1$ is non-compact.   Invoking Proposition~\ref{prop:RF->CompactNorm}, we deduce that $\Gamma$ is not residually Noetherian. 
\end{proof}

In the special case where $T$ is the $4$-regular tree,  $H = \Aut(T)$ and the local action of $\Gamma$ on $T$ is $C_2 \times C_2$, we recover \cite[Lemma~9.4]{JH16}.

\begin{proof}[Proof of Corollary~\ref{cor:WiseExample}]
	It follows from \cite[Theorem~5.3]{Wise} that $\Gamma$ contains an element which fixes a vertex in $T_6$ but whose image in $\Aut(T_6)$ generates an infinite group; see \cite[Section II.4]{Wise_PhD} and \cite[Prop.~9]{Rattaggi_AntiTori}. In particular, the projection of $\Gamma$ to $\Aut(T_6)$ is non-discrete. On the other hand, the image of $\Gamma$ in $\Aut(T_4)$ is vertex-transitive, and one verifies that it acts without inversion. Therefore, if the local action of $F$ of $\Gamma$ on $T_4$ is not a two group, it has exactly two orbits of size $1$ and $3$ respectively. This implies that the closure of the image of $\Gamma$ in $\Aut(T_4)$ is a strictly ascending HNN extension, hence it is not unimodular, which is absurd since $\Gamma$ is a lattice. 
	
	An alternative argument consists in computing the local action of $\Gamma$ on 
	$T_4$ directly from the presentation: it is easily verified to be  $C_2 \times C_2$ acting on $4$ points. This has been done by D.~Rattaggi: the Wise lattice is Example 2.36 in \cite{Rattaggi_PhD}, and its local action is recorded in the table on p.~280 in Section~C.5 of loc.~cit.  
	
	The hypotheses of Corollary~\ref{cor:NonRF} are fulfilled, and the conclusion follows. 
\end{proof}


\bibliographystyle{amsplain}
\bibliography{biblio}

\end{document}